\documentclass[final,a4paper]{siamltex}

\usepackage{amsmath,amsfonts,amssymb,mathtools}
\usepackage{graphicx,subfigure}
\usepackage{braket,dsfont,url,enumerate}

\usepackage{pdfpages} 
\usepackage{cmap}
\usepackage{hyperref}
\usepackage{color,dsfont,braket}
\usepackage{bbm}
\usepackage[headings]{fullpage}
\usepackage{fullpage}
\usepackage{fancyhdr}
\usepackage{times}
\usepackage{cmap}
\usepackage{pgf}
\usepackage{comment}
\usepackage{ulem}

\newcommand{\C}{\mathbb{C}}

\newcommand{\na}{\nabla}
\newcommand{\pa}{\partial}
\newcommand{\eps}{\varepsilon}
\newcommand{\om}{\omega}
\newcommand{\Om}{\Omega}

\newcommand{\si}{\sigma}

\newcommand{\IOm}{I\times \Om}
\newcommand{\ImOm}{I_m\times \Om}

\newcommand{\Id}{\operatorname{Id}}
\newcommand{\diam}{\operatorname{diam}}

\newcommand{\vertiii}[1]{{\left\vert\kern-0.25ex\left\vert\kern-0.25ex\left\vert #1
    \right\vert\kern-0.25ex\right\vert\kern-0.25ex\right\vert}}

\newcommand{\norm}[1]{\lVert#1\rVert}
\newcommand{\ltwonorm}[1]{\lVert#1\rVert_{L^2(\Omega)}}

\newcommand{\abs}[1]{\lvert#1\rvert}
\newcommand{\Ppol}[1]{\mathcal{P}_{#1}}

\newcommand{\half}{\frac{1}{2}}
\newcommand{\thalf}{\frac{3}{2}}

\newcommand{\R}{\mathbb{R}}
\newcommand{\lh}{\abs{\ln{h}}}
\newcommand{\lk}{\ln{\frac{T}{k}}}

\newcommand{\Xkh}{X^{q,r}_{k,h}}
\newcommand{\Th}{\mathcal{T}_h}

\definecolor{darkred}{rgb}{.7,0,0}

\definecolor{green}{rgb}{0,0.7,0}

\begin{document}


\title{Global and interior pointwise best approximation results for the gradient of Galerkin solutions for parabolic problems}
\author{
Dmitriy Leykekhman\footnotemark[2]
\and
Boris Vexler\footnotemark[3]
}

\pagestyle{myheadings}
\markboth{DMITRIY LEYKEKHMAN AND BORIS VEXLER}{Parabolic best approximation pointwise }

\maketitle

\renewcommand{\thefootnote}{\fnsymbol{footnote}}
\footnotetext[2]{Department of Mathematics,
               University of Connecticut,
              Storrs,
              CT~06269, USA (leykekhman@math.uconn.edu). }

\footnotetext[3]{Chair of Optimal Control, Faculty of Mathematics, Technical University of Munich, Boltzmannstra{\ss}e 3, 85748 Garching b. Munich, Germany
(vexler@ma.tum.de). }

\renewcommand{\thefootnote}{\arabic{footnote}}


\begin{abstract}
In this paper we establish best approximation property of fully discrete  Galerkin solutions of second order parabolic problems on convex polygonal and polyhedral domains in the $L^\infty(I;W^{1,\infty}(\Om))$ norm. The discretization method consists of continuous Lagrange finite elements in space and discontinuous Galerkin methods of arbitrary order in time. The method of the proof differs from the established fully discrete error estimate techniques and uses only elliptic results and discrete maximal parabolic regularity for discontinuous Galerkin methods established by the authors in \cite{LeykekhmanD_VexlerB_2016b}. In addition, the proof does not require any relationship between spatial mesh sizes  and time steps. We also establish interior best approximation property that shows more local dependence of the error at a point.
\end{abstract}

\begin{keywords}
optimal control, pointwise control, parabolic problems, finite elements, discontinuous Galerkin, error estimates, pointwise error estimates
\end{keywords}

\begin{AMS}
\end{AMS}
\section{Introduction}
Let $\Om$ be a convex polygonal/polyhedral domain in $\mathbb{R}^N$, $N=2,3$ and $I=(0,T)$ with some $T>0$. We consider a second order parabolic problem
\begin{equation}\label{eq: heat equation}
\begin{aligned}
u_t(t,x)-\Delta u(t,x) &= f(t,x), & (t,x) &\in \IOm,\;  \\
    u(t,x) &= 0,    & (t,x) &\in I\times\pa\Omega, \\
   u(0,x) &= u_0(x),    & x &\in \Omega.
\end{aligned}
\end{equation}

To discretize the problem we use  continuous Lagrange finite elements in space and discontinuous Galerkin methods in time. The precise description of the method is given in section \ref{sec: discretization}.
Our main goal in this  paper is to establish global and interior (local) space-time pointwise best approximation type results for the fully discrete error. The global estimate has the following structure:
\begin{equation}\label{eq: best approximation property}
\|\na(u-u_{kh})\|_{L^\infty(I\times \Om)}\le C\ell_{k}\ell_{h}\|\na(u-\chi)\|_{L^\infty(I\times\Om)},
\end{equation}
where $u_{kh}$ denotes the fully discrete solution and $\chi$ is an arbitrary element of the finite dimensional space, $h$ stands for spatial mesh size and $k$ for the maximal time step, and $\ell_{k}, \ell_h$ stand for some logarithmic terms. Such results are sometimes called symmetric estimates, cf.~\cite{ChrysafinosK_WalkingtonNJ_2006a, DupontTF_LiuY_2002}.
The interior (local) result provides an estimate of the error $|\nabla(u-u_{u_{kh}})(\tilde t, x_0)|$ for given $\tilde t \in (0,T]$ and $x_0 \in \Omega$ in terms of best approximation on a ball $B_d(x_0)$ and some global terms in weaker norms. Precise results are stated in section 2, see Theorem~\ref{thm:global_best_approx} and Theorem~\ref{thm:local_best_approx}. For the global estimate~\eqref{eq: best approximation property} we assume that  $f$ and $u_0$ are such that $\nabla u \in C(\bar I \times \bar \Omega)$. For the interior result we essentially need only $\nabla u \in C(\bar I \times \bar B_d(x_0)) \cap L^2(I \times \Omega)$.
Such best approximation type results have only natural assumptions on the problem data and are desirable in many applications, for example optimal control problems governed by parabolic equations with gradient constraints, cf.~\cite{LudoviciWollner:2015}. We refer to a recent paper \cite{TantardiniF_VeeserA_2016a} for a further discussion on the importance of best approximation results and difficulties associated with obtaining such estimates for parabolic problems.

For elliptic problems the best approximation property as ~\eqref{eq: best approximation property}, which is equivalent to the stability of the Ritz projection in $W^{1,\infty}(\Omega)$ norm, is well known. The first log-free result was established in~\cite{RRannacher_RScott_1982} on convex polygonal domains. Later  the result was extended to convex polyhedral domains with some restriction on angles in~\cite{BrennerSC_ScottLR_1994}. This restriction was removed in~\cite{JGuzman_DLeykekhman_JRossmann_AHSchatz_2009a} and even extended to certain graded meshes in~\cite{DemlowA_LeykekhmanD_SchatzAH_WahlbinLB_2010a}. For parabolic problems similar results are rather scarce. The main body of the work on pointwise error estimates for parabolic problems are devoted to $L^\infty(I\times \Om)$ error estimates, see~\cite{LeykekhmanD_VexlerB_2016d} for review of the corresponding results. We are aware of only three publications dealing with pointwise error estimates for the gradient of the error.

In two space dimensions, semidiscrete error estimates were studied in \cite{ChenH_1993} and the fully discrete Crank-Nicolson method was studied in \cite{ThomeeV_XuJ_ZhangNY_1989}. Since the main motivation of both investigations was the question of superconvergence of the gradient of the error, it was assumed that the solution is sufficiently smooth. More general fully discrete error estimates using Pad\'{e} time schemes  were obtained in  \cite{LeykekhmanD_WahlbinLB_2008a} for smooth domains in $\mathbb{R}^N$.

In both publications dealing with fully discrete error estimates, \cite{LeykekhmanD_WahlbinLB_2008a} and \cite{ThomeeV_XuJ_ZhangNY_1989},  the proofs are based on the  splitting $u-u_{kh}=(u-R_hu)+(R_hu-u_{kh})$, where $R_h$ is the Ritz projection. This idea was first introduced by M. Wheeler \cite{WheelerMF_1973a} in order to obtain optimal order error estimates in $L^2$ norm in space. The main idea of this approach is the following: the first part of the error is treated by elliptic results and the second part satisfies a certain parabolic equation with the right-hand side involving $(u-R_hu)$, which can be treated by results from rational approximation of analytic semigroups in Banach spaces (see also \cite[Thm~.8.6]{ThomeeV_2006}). However, this approach requires additional smoothness of the solution, well beyond the natural regularity $\nabla u \in C(\bar I \times \bar \Omega)$ of the exact solution. Our approach is completely different. It uses newly established discrete maximal parabolic regularity results~\cite{LeykekhmanD_VexlerB_2016b} for discontinuous Galerkin time schemes, see section~\ref{sec: discrete max} below, and the discrete resolvent estimate of the form:
\begin{equation}\label{eq: resolven estimate intro}
\vertiii{(z+\Delta_h)^{-1}\chi}\le  \frac{M_h}{|z|}\vertiii{\chi},\quad\text{for}\ z\in \mathbb{C}\setminus \Sigma_{\gamma},\quad \text{for all}\ \chi\in \mathbb{V}_h=V_h+iV_h,
\end{equation}
where $M_h$ may depend on $\abs{\ln{h}}$ but is independent of $h$ otherwise,
$V_h$ is the space of continuous Lagrange finite elements of degree $r$, $\Delta_h$ is the discrete Laplace operator, see~\eqref{eq:discreteLaplace} below, and
\[
\Sigma_\gamma= \Set{z \in \mathbb{C} | \abs{\arg{(z)}} \le \gamma},
\]
for some $\gamma\in (0,\frac{\pi}{2})$.
In~\cite{LeykekhmanD_VexlerB_2016d} we showed this estimate for the triple norm $\vertiii{v_h}=\|\sigma^{\frac{N}{2}}v_h\|_{L^2(\Om)}$ with the weight function  $\si(x)=\sqrt{|x-x_0|^2+K^2h^2}$. This norm behaves similar to the $L^1(\Omega)$ norm, and we used the corresponding discrete maximal parabolic result to prove (global and interior) pointwise best approximation for function values of the solution $u$. Here, we will in addition require the estimate~\eqref{eq: resolven estimate intro} with respect to the norm
$$
\vertiii{v_h}=\|\sigma^{\frac{N}{2}}\na\Delta_h^{-1}v_h\|_{L^2(\Om)},
$$
which behaves similar to the the $W^{-1,1}(\Omega)$ norm, see Theorem~\ref{thm: very weighted_Resolvent} below. This allows us to prove our main results of  (global and interior) pointwise best approximation for the gradient of the solution.

The rest of the paper is organized as follows. In the next section we describe the discretization method and state our main results. In section \ref{sec:elliptic}, we review some essential elliptic results in weighted norms. Section \ref{sec: weighted resolvent} is devoted to establishing the resolvent estimate in weighted norms. In section \ref{sec: discrete max}, we review our discrete maximal parabolic regularity result. Finally, in sections \ref{sec: proofs global results} and \ref{sec: proofs local results}, we provide proofs of the global and interior best approximation properties of the fully discrete solution.


\section{Discretization and statement of main results}\label{sec: discretization}
To introduce the time discontinuous Galerkin discretization for the problem,
 we partition  $(0,T]$ into subintervals $I_m = (t_{m-1}, t_m]$ of length $k_m = t_m-t_{m-1}$, where $0 = t_0 < t_1 <\cdots < t_{M-1} < t_M =T$. The maximal and minimal time steps are denoted by $k =\max_{m} k_m$ and $k_{\min}=\min_{m} k_m$, respectively.
We impose the following conditions on the time mesh (as in ~\cite{LeykekhmanD_VexlerB_2016b} or ~\cite{DMeidner_RRannacher_BVexler_2011a}):
\begin{enumerate}[(i)]
  \item There are constants $c,\beta>0$ independent of $k$ such that
    \[
      k_{\min}\ge ck^\beta.
    \]
  \item There is a constant $\kappa>0$ independent of $k$ such that for all $m=1,2,\dots,M-1$
    \[
    \kappa^{-1}\le\frac{k_m}{k_{m+1}}\le \kappa.
    \]
  \item It holds $k\le\frac{1}{4}T$.
\end{enumerate}
The semidiscrete space $X_k^q$ of piecewise polynomial functions in time is defined by
\[
X_k^q=\Set{v_{k}\in L^2(I;H^1_0(\Om)) | v_{k}|_{I_m}\in \Ppol{q}(I_m;H^1_0(\Om)), \ m=1,2,\dots,M },
\]
where $\Ppol{q}(I_m;V)$ is the space of polynomial functions of degree $q$ in time with values in a Banach space $V$.
We will employ the following notation for time dependent functions 
\begin{equation}\label{def: time jumps}
v^+_m=\lim_{\eps\to 0^+}v(t_m+\eps), \quad v^-_m=\lim_{\eps\to 0^+}v(t_m-\eps), \quad [v]_m=v^+_m-v^-_m,
\end{equation}
if these limits exist.
Next we define the following bilinear form
\begin{equation}\label{eq: bilinear form B}
 B(v,\varphi)=\sum_{m=1}^M \langle v_t,\varphi \rangle_{I_m \times \Omega} + (\na v,\na \varphi)_{\IOm}+\sum_{m=2}^M([v]_{m-1},\varphi_{m-1}^+)_\Om+(v_{0}^+,\varphi_{0}^+)_\Om,
\end{equation}
where $( \cdot,\cdot )_{\Omega}$ and $( \cdot,\cdot )_{I_m \times \Omega}$ are the usual $L^2$ space and space-time inner-products,
$\langle \cdot,\cdot \rangle_{I_m \times \Omega}$ is the duality product between $ L^2(I_m;H^{-1}(\Omega))$ and $ L^2(I_m;H^{1}_0(\Omega))$. We note, that the first sum vanishes for $v \in X^0_k$. 
Rearranging the terms in \eqref{eq: bilinear form B}, we obtain an equivalent (dual) expression of $B$:
\begin{equation}\label{eq:B_dual}
 B(v,\varphi)= - \sum_{m=1}^M \langle v,\varphi_t \rangle_{I_m \times \Omega} + (\na v,\na \varphi)_{\IOm}-\sum_{m=1}^{M-1} (v_m^-,[\varphi]_m)_\Om + (v_M^-,\varphi_M^-)_\Om.
\end{equation}

To introduce the fully discrete approximation, let $\Th$ for $h>0$  denote  a quasi-uniform triangulation of $\Om$  with mesh size $h$, i.e., $\Th = \{\tau\}$ is a partition of $\Om$ into cells (triangles or tetrahedrons) $\tau$ of diameter $h_\tau$ such that for $h=\max_{\tau} h_\tau$,
$$
\operatorname{diam}(\tau)\le h \le C |\tau|^{\frac{1}{N}}, \quad \text{for all }\; \tau\in \Th,
$$
hold. Let $V_h$ be the set of all functions in $H^1_0(\Om)$ that are polynomials of degree $r$ on each $\tau$, i.e. $V_h$ is the usual space of conforming finite elements.
To obtain the fully discrete approximation we consider the space-time finite element space
\begin{equation} \label{def: space_time}
\Xkh=\Set{v_{kh} \in L^2(I;H^1_0(\Omega)) | v_{kh}|_{I_m}\in \Ppol{q}(I_m;V_h), \ m=1,2,\dots,M}, \quad q\geq 0,\quad r\geq 1.
\end{equation}
We define a fully discrete dG($q$)cG($r$) solution $u_{kh} \in \Xkh$ by
\begin{equation}\label{eq:fully discrete heat with RHS}
B(u_{kh},\varphi_{kh})=(f,\varphi_{kh})_{\IOm}+(u_0,\varphi_{kh,0}^+)_\Om \quad \text{for all }\; \varphi_{kh}\in \Xkh.
\end{equation}

\subsection{Main results}

Now we state our main results.
The first result establishes the global best approximation property of the fully discrete Galerkin solution in the $L^\infty(I;W^{1,\infty}(\Om))$ norm.

\begin{theorem}[Global best approximation]\label{thm:global_best_approx}
Let $u$ and $u_{kh}$ satisfy \eqref{eq: heat equation} and \eqref{eq:fully discrete heat with RHS}
respectively. Then, there exists a constant $C$ independent of $k$ and  $h$ such that
\[
\norm{\na(u-u_{kh})}_{L^\infty(I\times \Om)} \le C \ell_k \ell_h \inf_{\chi \in \Xkh} \norm{\na (u-\chi)}_{L^\infty(I\times\Om)},
\]
where $\ell_k = \lk$ and $\ell_h=\lh^{\frac{2N-1}{N}}$.
\end{theorem}

The proof of this theorem is given in Section~\ref{sec: proofs global results}.

For the error at a given point $x_0\in \Omega$ we obtain a sharper results. For elliptic problems similar results were obtained in~\cite{AHSchatz_LBWahlbin_1977a, AHSchatz_LBWahlbin_1995a}. We denote by $B_d=B_d(x_0)$ the  ball of radius $d$ centered at $x_0$.
\begin{theorem}[Interior best approximation]\label{thm:local_best_approx}
Let $u$ and $u_{kh}$ satisfy \eqref{eq: heat equation} and \eqref{eq:fully discrete heat with RHS}, respectively and let  $d>4h$. Assume $x_0\in \Om$ and $\tilde{t}\in I_m$ for some $m=1,2,\dots, M$ and $B_d\subset\subset\Om$. Then there exists a constant $C$ independent of $h$, $k$, and $d$ such that
$$
\begin{aligned}
|\na(u-u_{kh})(\tilde{t},x_0)|&\le C \ell_k \ell_h\inf_{\chi \in \Xkh}\Biggl\{\|\na(u-\chi)\|_{L^\infty((0,t_m)\times B_d(x_0))}+d^{-1}\|u-\chi\|_{L^\infty((0,t_m)\times B_d(x_0))}\\
 &+d^{-\frac{N}{2}}\bigg(\|\na(u-\chi)\|_{L^\infty((0,t_m);L^2(\Om))}+d^{-1}\|u-\chi\|_{L^\infty((0,t_m);L^2(\Om))}\bigg)\Biggr\},
\end{aligned}
$$
with $\ell_k$ and $\ell_h$ defined as in Theorem~\ref{thm:global_best_approx}.
\end{theorem}

The proof of this theorem is given in Section~\ref{sec: proofs local results}.

\section{Elliptic estimates in weighted norms}\label{sec:elliptic}

In this section we collect some estimates for the finite element discretization of elliptic problems in weighted norms on convex polygonal/polyhedral domains mainly taken from~\cite{LeykekhmanD_VexlerB_2016c} . These results will be used in the following sections within the proofs of Theorem~\ref{thm: very weighted_Resolvent}, Theorem~\ref{thm:global_best_approx}, and Theorem~\ref{thm:local_best_approx}.

In this section we consider a fixed (but arbitrary) point $x_0 \in \Omega$. Associated to this point we introduce a smoothed delta function \cite[~Appendix]{AHSchatz_LBWahlbin_1995a}, which we will denote by $\tilde{\delta}$. This function is supported in one cell, which is denoted by $\tau_{0}$ with $x_0 \in \bar \tau_0$, and satisfies
\begin{equation}\label{eq: definition delta}
(\chi, \tilde{\delta})_{\tau_{0}}=\chi({x_0}), \quad \text{for all }\; \chi\in \Ppol{r}(\tau_{0}).
\end{equation}
In addition we also have, see, e.g.,~\cite[Lemma~2.2]{ThomeeV_WahlbinLB_2000a},
\begin{equation}\label{delta1}
 \|\tilde{\delta}\|_{W^{s,p}(\Om)} \le C h^{-s-N(1-\frac{1}{p})}, \quad 1\le p \le \infty, \quad s=0,1,2.
\end{equation}
Thus in particular $\|\tilde{\delta}\|_{L^1(\Om)} \le C$, $\norm{\tilde{\delta}}_{L^2(\Om)} \le Ch^{-\frac{N}{2}}$, and  $\|\tilde{\delta}\|_{L^\infty(\Om)} \le Ch^{-N}$.
Next we introduce a weight function
\begin{equation}\label{eq: sigma weight}
\sigma(x) = \sqrt{|x-x_0|^2+K^2h^2},
\end{equation}
where $K>0$ is a sufficiently large constant.  This weight function was first introduced in \cite{NattererF_1975, NitscheJA_1971a} to analyze pointwise finite element error estimates.
One can easily check that $\sigma$ satisfies the following properties:
\begin{subequations}
\begin{align}
\norm{\sigma^{-\frac{N}{2}}}_{L^2(\Om)}&\le C\lh^{\frac{1}{2}}, \label{eq: property 1 of sigma}\\
|\na \sigma|&\le C, \label{eq: property 2 of sigma}\\
|\na^2 \sigma|&\le C\sigma^{-1}, \label{eq: property 3 of sigma}\\
\max_\tau{\sigma}&\le C\min_\tau{\sigma} \quad \text{for all }\; \tau\in \Th \label{eq: property 4 of sigma}.
\end{align}
\end{subequations}
For the finite element space $V_h$ we will utilize the $L^2$ projection $P_h \colon L^2(\Omega) \to V_h$ defined by
\begin{equation}\label{eq:l2_proj}
(P_hv,\chi)_{\Om} = (v,\chi)_{\Om} \quad \text{for all }\; \chi\in V_h,
\end{equation}
the Ritz projection $R_h \colon H^1_0(\Omega) \to V_h$ defined by
\begin{equation}\label{eq:Ritz_proj}
(\nabla R_hv,\nabla \chi)_{\Om} = (\nabla v,\nabla \chi)_{\Om} \quad \text{for all }\; \chi\in V_h,
\end{equation}
and the usual nodal interpolation operator $i_h \colon C_0(\Omega) \to V_h$  with usual approximation properties (cf., e.\,g.,~\cite[Theorem~3.1.5]{PGCiarlet_1978a})
\begin{equation}\label{eq:I_h approximation}
\|u-i_hu\|_{L^q(\Om)}\le Ch^{2+N(\frac{1}{q}-\frac{1}{p})}\|u\|_{W^{2,p}(\Om)}, \quad \text{for}\quad q\geq p > \frac{N}{2},
\end{equation}
as well as the Scott-Zhang interpolation operator $i^{SZ}_h \colon W^{1,1}_0(\Omega) \to V_h$
with the approximation properties (cf., e.\,g.,~\cite{ScottLR_ZhangS_1990}) for $N=3$:
\begin{equation}\label{eq:I_h approximation SZ}
h\|\nabla(u-i^{SZ}_hu)\|_{L^2(\Om)} + \|u-i^{SZ}_hu\|_{L^2(\Om)}\le Ch^{\frac{3}{2}}\|u\|_{W^{2,\frac{3}{2}}(\Om)}\quad \text{for all } u \in W^{2,\frac{3}{2}}(\Omega) \cap W^{1,1}_0(\Omega).
\end{equation}
Moreover we introduce the discrete Laplace operator $\Delta_h \colon V_h \to V_h$ defined by
\begin{equation}\label{eq:discreteLaplace}
(-\Delta_h v_h,\chi)_{\Om} = (\nabla v_h,\nabla \chi)_{\Om}, \quad \text{for all }\; \chi\in V_h.
\end{equation}

The next lemma states an approximation result for the Ritz projection in the $L^\infty(\Omega)$ norm.
\begin{lemma}\label{lemma: approximation of Ritz}
There exists a constant $C>0$ independent of $h$, such that
$$
\|v-R_hv\|_{L^\infty(\Om)}\le Ch\lh\|\na v\|_{L^\infty(\Om)}.
$$
\end{lemma}
 For smooth domains such a result was established in~\cite{NitscheJ_1975a, AHSchatz_LBWahlbin_1977a, AHSchatz_LBWahlbin_1982a} (logfree for higher order elements), for polygonal domains in~\cite{AHSchatz_1980a} and~\cite[Theorem 3.2]{ErikssonK_1994a} (for mildly graded meshes), and for convex polyhedral domains it follows  from stability of the Ritz projection in the $L^\infty(\Omega)$ norm in~\cite[Theorem 12]{LeykekhmanD_VexlerB_2016c}.

The following lemma is a superapproximation result in weighted norms.
\begin{lemma}[Lemma 3 in~\cite{LeykekhmanD_VexlerB_2016c}]\label{lemma:super_ih_ph}
Let $v_h \in V_h$. Then the following estimates hold for any $\alpha,\beta \in \R$ and $K$ (in the definition~\eqref{eq: sigma weight} of the weight $\sigma$) large enough:
\begin{subequations}
\begin{equation}\label{sigma_est_ih3_ih4}
\norm{\sigma^\alpha(\Id-i_h)(\sigma^\beta v_h)}_{L^2(\Om)} + h \norm{\sigma^\alpha\nabla(\Id-i_h)(\sigma^\beta v_h)}_{L^2(\Om)} \le c h \norm{\sigma^{\alpha+\beta -1} v_h}_{L^2(\Om)},
\end{equation}
\begin{equation}\label{sigma_est_ph3_ph4}
\norm{\sigma^\alpha(\Id-P_h)(\sigma^\beta v_h)}_{L^2(\Om)}  + h \norm{\sigma^\alpha\nabla(\Id-P_h)(\sigma^\beta v_h)}_{L^2(\Om)} \le c h \norm{\sigma^{\alpha+\beta -1} v_h}_{L^2(\Om)}.
\end{equation}
\end{subequations}
\end{lemma}

The next lemma describes a connection between the regularized delta function $\tilde \delta$ and the weight $\sigma$.
\begin{lemma}\label{lemma:sigma_delta}
There hold
\begin{equation}\label{eq:sigma_delta}
\norm{\sigma^{\frac{N}{2}} \tilde \delta}_{L^2(\Om)} + \norm{\sigma^{\frac{N+2}{2}} \na\tilde \delta}_{L^2(\Om)} + h\norm{\sigma^{\frac{N}{2}} \na\tilde \delta}_{L^2(\Om)} \le C
\end{equation}
and
\begin{equation}\label{eq:sigma_delta_Ph}
 \norm{\sigma^{\frac{N}{2}} P_h\tilde \delta}_{L^2(\Om)} + \norm{\sigma^{\frac{N+2}{2}} P_h \na\tilde \delta}_{L^2(\Om)}+ h\norm{\sigma^{\frac{N}{2}} P_h \na\tilde \delta}_{L^2(\Om)}\le C.
\end{equation}
\end{lemma}
The proof of the first two terms in~\eqref{eq:sigma_delta} and~\eqref{eq:sigma_delta_Ph} respectively can be found in~\cite{ErikssonK_JohnsonC_1995a} for $N=2$ and in~\cite[Lemma 4]{LeykekhmanD_VexlerB_2016c} for $N=3$. Using similar arguments it is straightforward to show the result for the other terms.

The following two lemmas provide the flexibility in manipulating weighted norms.

\begin{lemma}\label{lemma:sigma_nabla}
For each $\alpha \in \R$, there is a constant $C>0$ such that for any $v \in H^1_0(\Omega)\cap H^2(\Omega)$ there holds
\[
\ltwonorm{\sigma^\alpha \nabla v} \le C \left( \ltwonorm{\sigma^{\alpha+1}\Delta v} + \ltwonorm{\sigma^{\alpha-1}v}\right).
\]
\end{lemma}
\begin{proof}
There holds
\[
\begin{aligned}
\ltwonorm{\sigma^\alpha \nabla v}^2 &= (\sigma^{2\alpha} \nabla v, \nabla v)
= (\nabla(\sigma^{2\alpha} v),\nabla v) - 2 \alpha (v \sigma^{2\alpha-1}\nabla \sigma, \nabla v)\\
&= -(\sigma^{\alpha-1} v, \sigma^{\alpha + 1}\Delta v) - 2 \alpha (v \sigma^{\alpha-1}\nabla \sigma, \sigma^{\alpha}\nabla v).
\end{aligned}
\]
Using $\abs{\nabla \sigma} \le C$ we obtain
\[
\ltwonorm{\sigma^\alpha \nabla v}^2 \le \ltwonorm{\sigma^{\alpha-1} v} \ltwonorm{\sigma^{\alpha + 1}\Delta v} + C \ltwonorm{\sigma^{\alpha-1} v} \ltwonorm{\sigma^\alpha \nabla v}.
\]
Absorbing $\ltwonorm{\sigma^\alpha \nabla v}$ we obtain the desired estimate.
\end{proof}

\begin{lemma}\label{lemma:sigma_nabla_h}
For each $\alpha \in \R$, there is a constant $C>0$ such that for any $v_h \in V_h$ there holds
\[
\ltwonorm{\sigma^\alpha \nabla v_h} \le C \left( \ltwonorm{\sigma^{\alpha+1}\Delta_h v_h} + \ltwonorm{\sigma^{\alpha-1}v_h}\right).
\]
\end{lemma}
\begin{proof}
Similar to the proof of the previous lemma we have
\[
\begin{aligned}
\ltwonorm{\sigma^\alpha \nabla v_h}^2 &= (\sigma^{2\alpha} \nabla v_h, \nabla v_h)
= (\nabla(\sigma^{2\alpha} v_h),\nabla v_h) - 2 \alpha (v_h \sigma^{2\alpha-1}\nabla \sigma, \nabla v_h)\\
&= (\nabla P_h(\sigma^{2\alpha} v_h),\nabla v_h) + (\nabla(Id -P_h)(\sigma^{2\alpha} v_h),\nabla v_h) - 2 \alpha (v_h \sigma^{2\alpha-1}\nabla \sigma, \nabla v_h)\\
&= -(\sigma^{\alpha-1} v_h, \sigma^{\alpha + 1}\Delta_h v_h) + (\sigma^{-\alpha} \nabla(Id -P_h)(\sigma^{2\alpha} v_h),\sigma^{\alpha} \nabla v_h) - 2 \alpha (v_h \sigma^{\alpha-1}\nabla \sigma, \sigma^{\alpha}\nabla v_h).
\end{aligned}
\]
Applying Lemma~\ref{lemma:super_ih_ph} for the second term and using $\abs{\nabla \sigma} \le C$ we obtain
\[
\ltwonorm{\sigma^\alpha \nabla v_h}^2 \le \ltwonorm{\sigma^{\alpha-1} v_h} \ltwonorm{\sigma^{\alpha + 1}\Delta_h v_h} + C \ltwonorm{\sigma^{\alpha-1} v_h} \ltwonorm{\sigma^\alpha \nabla v_h}.
\]
Absorbing $\ltwonorm{\sigma^\alpha \nabla v_h}$ we obtain the desired estimate.
\end{proof}

In the following proofs we will make a heavy use of pointwise estimates for the Green's function.
\begin{lemma}                          \label{lemma: Green}
Let $G(x,y)$ denotes the elliptic Green's function of the Laplace operator on the domain $\Omega$. Then for $N=2,3$ the following estimates hold,
\begin{subequations}         \label{eq: Greens}
\begin{align}
|\na_x G(x,y)| &\le C |x-y|^{1-N},\quad \text{for all }\; x,y\in\Om,\quad x\neq y,\\
|\na_y G(x,y)| &\le C|x-y|^{1-N}, \quad \text{for all }\; x,y\in\Om,\quad x\neq y.\\
|\na_y \na_x G(x,y)| & \le C|x-y|^{-N}, \quad \text{for all }\; x,y\in\Om,\quad x\neq y.\label{eq:est3_green}
\end{align}
\end{subequations}
\end{lemma}
The proof of the first estimate  can be found in \cite[Prop~1]{FrommSJ_1993a} and the second one follows from the symmetry of the Green's function and the first estimate, i.e. $|\na_y G(x,y)|=|\na_x G(y,x)|\le  C |x-y|^{1-N}$. The third estimate is also proven in~\cite[Prop~1]{FrommSJ_1993a}.

The next lemma can be thought of as weighted Gagliardo-Nirenberg interpolation inequality.
\begin{lemma}[Lemma~5 in~\cite{LeykekhmanD_VexlerB_2016c}]\label{lemma: from weighted L2 to weighted in H1}
Let $N=3$. There exists a constant $C$ independent of $K$ and $h$ such that for any $f\in H^1_0(\Om)$, any $\alpha,\beta \in \mathbb{R}$ with $\alpha \ge -\frac{1}{2}$ and any $1\le p\le \infty$, $\frac{1}{p}+\frac{1}{p'}=1$ there holds:
$$
\|\sigma^{\alpha}f\|^2_{L^2(\Om)}\le C\|\sigma^{\alpha-\beta}f\|_{L^p(\Om)} \|\sigma^{\alpha + 1 +\beta}\na f\|_{L^{p'}(\Om)},
$$
provided $\|\sigma^{\alpha-\beta}f\|_{L^p(\Om)}$ and  $\|\sigma^{\alpha + 1 +\beta}\na f\|_{L^{p'}(\Om)}$ are bounded.
\end{lemma}

\begin{lemma}\label{lemma: inverse_delta derivative}
Let $D=\pa_{x_i}$, $i=1,\dots,N$ denote any partial derivative. Then for $N=2,3$ there holds
\begin{equation}\label{eq:inverse_delta derivative}
\norm{ \si^{\frac{N-2}{2}}\Delta^{-1}D\tilde \delta}_{L^2(\Om)} +\norm{ \si^{\frac{N}{2}}\na\Delta^{-1}D\tilde \delta}_{L^2(\Om)} \le C\lh^{\frac{1}{2}}
\end{equation}
and for $N=3$ there holds
\begin{equation}\label{eq:g_l3_w132}
\norm{\Delta^{-1}D\tilde \delta}_{L^3(\Omega)}  + \norm{\nabla \Delta^{-1}D\tilde \delta}_{L^{\frac{3}{2}}(\Omega)} \le C h^{-1}.
\end{equation}
\end{lemma}
\begin{proof}
Consider the following elliptic problem
\begin{equation}\label{eq: elliptic derivative greens equation}
\begin{aligned}
-\Delta g(x) &= D\tilde \delta(x), & x &\in \Om,\;  \\
    g(x) &= 0,    & x &\in \pa\Omega.
\end{aligned}
\end{equation}
Thus, in order to obtain the estimate~\eqref{eq:inverse_delta derivative} we need to establish
$$
\norm{\si^{\frac{N-2}{2}}g}_{L^2(\Om)}+\norm{\si^{\frac{N}{2}}\na g}_{L^2(\Om)}\le C\lh^{\frac{1}{2}}.
$$
To estimate the first term, we will be using the following Green's function representation
\begin{equation}\label{eq: representation Greens derivative}
g(x)=\int_{\tau_0}G(x,y) \partial_{y_i}\tilde \delta(y)\ dy=-\int_{\tau_0} \partial_{y_i} G(x,y)\tilde \delta(y)\ dy.
\end{equation}
Define $B_h=B_{3h}(x_0)\cap\Om$ and $B^c_h=\Om\backslash B_h$ and consider two cases: $x\in B_h$ and  $x\in B^c_h$. In the case $x\in B_h$, we obtain using polar coordinates centered at $x$ and using ~\eqref{delta1}, ~\eqref{eq: representation Greens derivative}, and Lemma \ref{lemma: Green},
$$
\abs{g(x)}\le \|\tilde{\delta}\|_{L^\infty(\tau_0)}\int_{\tau_0}|\na_y G(x,y)|dy\le Ch^{-N}\int_{\tau_0}|x-y|^{1-N}\ dy\le Ch^{-N}\int_0^{ch} d\rho\le Ch^{1-N}.
$$
Hence by the H\"{o}lder inequality and using that $\si\le Ch$ on $B_h$, we have
$$
\norm{\si^{\frac{N-2}{2}} g}_{L^2(B_{h})}\le Ch^{\frac{N}{2}}h^{\frac{N-2}{2}}\norm{g}_{L^\infty(B_{h})}\le C.
$$
In the case $x\in B^c_h$, we have for any $y \in \tau_0$ by the triangle inequality
\[
\abs{x-y} \ge \abs{x-x_0} - \abs{y-x_0} \ge  \abs{x-x_0} -h
\]
and therefore again by~\eqref{eq: representation Greens derivative} and Lemma \ref{lemma: Green}
\[
\abs{g(x)}\le \|\tilde{\delta}\|_{L^1(\tau_0)}\frac{C}{(|x-x_0|-h)^{N-1}}\le \frac{C}{(|x-x_0|-h)^{N-1}}.
\]
Hence, using polar coordinates with $\rho=|x-x_0|$, we obtain
$$
\|\si^{\frac{N-2}{2}} g\|^2_{L^2(B^c_h)}\le C\int_{B^c_h}\frac{(|x-x_0|+Kh)^{N-2}}{(|x-x_0|-h)^{2N-2}}\ dx\le C\int_{3h}^{\diam(\Om)}\frac{(\rho+Kh)^{N-2}}{(\rho-h)^{2N-2}}\rho^{N-1}\ d\rho\le C\lh.
$$
Thus, we established
\begin{equation}\label{eq: esimate for si N g}
\norm{\si^{\frac{N-2}{2}} g}_{L^2(\Om)}\le C\lh^{\frac{1}{2}}.
\end{equation}
To estimate the second term in~\eqref{eq:inverse_delta derivative} we apply Lemma~\ref{lemma:sigma_nabla} and obtain
\[
\ltwonorm{\sigma^{\frac{N}{2}}\nabla g} \le C \left( \norm{\sigma^{\frac{N+2}{2}} D \tilde \delta} +  \norm{\sigma^{\frac{N-2}{2}} g}\right) \le C + C \norm{\sigma^{\frac{N-2}{2}} g} \le  C\lh^{\frac{1}{2}},
\]
where we have used Lemma~\ref{lemma:sigma_delta} and \eqref{eq: esimate for si N g}.

The first term in~\eqref{eq:g_l3_w132} is estimated as follows.
 There holds
\[
\norm{g}_{L^3(\Omega)}^3 = \norm{g}_{L^3(B_h)}^3 + \norm{g}_{L^3(B_h^c)}^3.
\]
For the term on $B_h$ we obtain as above
\[
\norm{g}_{L^3(B_h)}^3 \le C h^3 \norm{g}_{L^\infty(B_h)}^3 \le C h^{-3}.
\]
For the second term we have
\[
\norm{g}_{L^3(B_h^c)}^3 \le C\int_{B^c_h}\frac{1}{(|x-x_0|-h)^{6}}\ dx\le C\int_{3h}^{\diam(\Om)}\frac{1}{(\rho-h)^{6}}\rho^{2}\ d\rho\le Ch^{-3}.
\]

In order to estimate $\norm{\na g}_{L^\frac{3}{2}(\Omega)}$ we use the pointwise representation
\begin{equation}\label{eq:Green_repr_grad}
\na g(x) = \int_{\tau_0} \nabla_x G(x,y) \partial_{y_i} \tilde \delta (y) \, dy,
\end{equation}
apply Lemma~\ref{lemma: Green}, and obtain for $x \in B_h$
\[
\abs{\nabla g(x)}\le \|\tilde{\delta}\|_{W^{1,\infty}(\tau_0)}\int_{\tau_0}|\na_x G(x,y)|dy\le Ch^{-N-1}\int_{\tau_0}|x-y|^{1-N}\ dy\le Ch^{-N-1}\int_0^{ch} d\rho\le Ch^{-N}.
\]
Hence, for $N=3$, we have
\[
\norm{\na g}_{L^\frac{3}{2}(B_h)} \le C h^{-3} \left(h^3\right)^{\frac{2}{3}} = C h^{-1}.
\]
For $x \in B_h^c$ we integrate by parts in~\eqref{eq:Green_repr_grad},
\[
\na g(x) =  - \int_{\tau_0} \partial_{y_i} \nabla_x G(x,y) \tilde \delta (y) \, dy,
\]
and obtain using estimate~\eqref{eq:est3_green} from Lemma~\ref{lemma: Green}
\[
\abs{\na g(x)} \le \|\tilde{\delta}\|_{L^1(\tau_0)}\frac{C}{(|x-x_0|-h)^3}\le \frac{C}{(|x-x_0|-h)^3}.
\]
Thus,
\[
\norm{\na g}_{L^\frac{3}{2}(B_h^c)}^\frac{3}{2} \le C\int_{B^c_h}\frac{1}{(|x-x_0|-h)^{\frac{9}{2}}}\ dx
\le C\int_{3h}^{\diam(\Om)}\frac{\rho^2}{(\rho-h)^{\frac{9}{2}}}\, d\rho\le C h^{-\frac{3}{2}}.
\]
This completes the proof.
\end{proof}

We will also require a discrete version of the Lemma \ref{lemma: inverse_delta derivative}.
\begin{lemma}\label{lemma: Delta_h-1 for Di delta}
For $N=2,3$, we have
$$
\|\si^{\frac{N}{2}}\na\Delta_h^{-1}P_hD\tilde \delta\|_{L^2(\Om)}\le C\lh^{\frac{1}{2}}.
$$
\end{lemma}
\begin{proof}
Let $g$ be solution of \eqref{eq: elliptic derivative greens equation} and let $g_h\in V_h$ satisfy
\begin{equation}\label{eq: L discrete elliptic greens}
    -\Delta_h g_h = P_hD\tilde{\delta}.
\end{equation}
Notice that $g_h=R_hg$. Thus in order to establish the lemma, we need to show
$$
\|\si^{\frac{N}{2}}\na g_h\|_{L^2(\Om)}\le C\lh^{\frac{1}{2}}.
$$
For $N=2$ we apply Lemma~\ref{lemma:sigma_nabla_h} and obtain
\[
\|\si\na g_h\|_{L^2(\Om)} \le C \left( \ltwonorm{\si^2 P_h D \tilde \delta} + \norm{g_h}_{L^2(\Om)} \right) \le C + C \norm{g_h}_{L^2(\Om)},
\]
where we have used Lemma~\ref{lemma:sigma_delta}. Thus, for $N=2$ it remains to prove
\[
\|g_h\|_{L^2(\Om)}\le C\lh^{\frac{1}{2}}.
\]
To prove this estimate, we use  Lemma \ref{lemma: inverse_delta derivative}, global error estimates in the $L^2(\Omega)$, the $H^2$ regularity, and the property~\eqref{delta1} of $\tilde \delta$. Thus, we obtain
$$
\begin{aligned}
\|g_h\|_{L^2(\Om)}\le\|g\|_{L^2(\Om)}+\|g-g_h\|_{L^2(\Om)}&\le C\lh^{\frac{1}{2}}+Ch^2\|g\|_{H^2(\Om)}\\
&\le C\lh^{\frac{1}{2}}+Ch^2\|D\tilde \delta\|_{L^2(\Om)}\le C\lh^{\frac{1}{2}}.
\end{aligned}
$$
The case $N=3$ is more challenging.
By the triangle inequality we get
\begin{equation}\label{eq: after lemma}
\|\si^{\frac{3}{2}} \nabla g_h\|_{L^2(\Om)} \le \|\si^{\frac{3}{2}} \nabla g\|_{L^2(\Om)} + \|\si^{\frac{3}{2}} \nabla(g-g_h)\|_{L^2(\Om)}.
\end{equation}
For the first term we have by Lemma \ref{lemma: inverse_delta derivative}
$$
\|\si^{\frac{3}{2}}\nabla g\|_{L^2(\Om)}\le C\lh^{\frac{1}{2}}.
$$
For the second term we apply~\cite[Lemma 10]{LeykekhmanD_VexlerB_2016c}, which gives
\[
\|\si^{\frac{3}{2}} \nabla(g-g_h)\|_{L^2(\Om)} \le C h \left(\norm{\si^{\frac{3}{2}} \Delta_h g_h}_{L^2(\Om)} + \norm{\si^{\frac{1}{2}}\nabla g_h}_{L^2(\Om)}\right).
\]
For the term $\norm{\si^{\frac{3}{2}} \Delta_h g_h}_{L^2(\Om)}$ we get by Lemma~\ref{lemma:sigma_delta}
\[
\norm{\si^{\frac{3}{2}} \Delta_h g_h}_{L^2(\Om)} = \norm{\si^{\frac{3}{2}} P_h D \tilde \delta}_{L^2(\Om)} \le C h^{-1}.
\]
Inserting this estimate into~\eqref{eq: after lemma} we obtain
\begin{equation}\label{eq:s32n_hs12n}
\|\si^{\frac{3}{2}} \nabla g_h\|_{L^2(\Om)} \le  C\lh^{\frac{1}{2}} + C h \norm{\si^{\frac{1}{2}}\nabla g_h}_{L^2(\Om)}.
\end{equation}
Thus, it remains to estimate $\norm{\si^{\frac{1}{2}}\nabla g_h}_{L^2(\Om)}$.
To this end we apply Lemma~\ref{lemma:sigma_nabla_h} and obtain
\[
\norm{\si^{\frac{1}{2}}\nabla g_h}_{L^2(\Om)} \le C \left(\ltwonorm{\si^{\frac{3}{2}}P_hD \tilde \delta} + \ltwonorm{\si^{-\frac{1}{2}} g_h} \right).
\]
Using Lemma~\ref{lemma:sigma_delta} we obtain
\begin{equation}\label{eq: from weighted H1 to L2 2}
\|\si^{\frac{1}{2}}\na g_h\|_{L^2(\Om)}\le  Ch^{-1} + C\|\si^{-\frac{1}{2}} g_h\|_{L^2(\Om)}.
\end{equation}
To estimate $\|\si^{-\frac{1}{2}} g_h\|_{L^2(\Om)}$ we use Lemma \ref{lemma: from weighted L2 to weighted in H1}, with $\alpha=\beta=-\frac{1}{2}$ and $p=3$, to obtain
\begin{equation}\label{eq: from L2 using lemma}
\|\si^{-\half} g_h\|_{L^2(\Om)}\le C\|g_h\|^{\frac{1}{2}}_{L^3(\Om)}\|\na g_h\|^{\frac{1}{2}}_{L^{\frac{3}{2}}(\Om)} \le C\|g_h\|^{\frac{1}{2}}_{L^3(\Om)}\|\na g\|^{\frac{1}{2}}_{L^{\frac{3}{2}}(\Om)},
\end{equation}
where in the last step we used stability of the Ritz projection in the $W^{1,\frac{3}{2}}(\Omega)$ seminorm, see~\cite{JGuzman_DLeykekhman_JRossmann_AHSchatz_2009a}.
Using the inverse and the triangle inequalities,
$$
\begin{aligned}
\|g_h\|_{L^3(\Om)}&\le \|g\|_{L^3(\Om)}+\|g-g_h\|_{L^3(\Om)}\le \|g\|_{L^3(\Om)}+\|i_hg-g_h\|_{L^3(\Om)}+\|g-i_hg\|_{L^3(\Om)}\\
&\le \|g\|_{L^3(\Om)}+Ch^{-\frac{1}{2}}\|i_hg-g_h\|_{L^2(\Om)}+\|g-i_hg\|_{L^3(\Om)}\\
&\le \|g\|_{L^3(\Om)}+Ch^{-\frac{1}{2}}\|g-g_h\|_{L^2(\Om)}+Ch^{-\frac{1}{2}}\|g-i_hg\|_{L^2(\Om)}+\|g-i_hg\|_{L^3(\Om)}.
\end{aligned}
$$
Using the approximation theory \eqref{eq:I_h approximation}, the standard $L^2$ estimate, and the properties of $\tilde{\delta}$ function, we have
\begin{equation}\label{eq: approximation terms g-Ig}
h^{-\frac{1}{2}}\|g-g_h\|_{L^2(\Om)} + h^{-\frac{1}{2}}\|g-i_hg\|_{L^2(\Om)}+\|g-i_hg\|_{L^3(\Om)}\le Ch^{\frac{3}{2}}\|g\|_{H^2(\Om)}\le Ch^{\frac{3}{2}}\|D\tilde{\delta}\|_{L^2(\Om)}\le Ch^{-1}.
\end{equation}
By Lemma~\ref{lemma: inverse_delta derivative} we have
\[
\norm{g}_{L^3(\Omega)}  + \norm{\nabla g}_{L^{\frac{3}{2}}(\Omega)} \le C h^{-1}.
\]
Inserting this in~\eqref{eq: from L2 using lemma} and~\eqref{eq: from weighted H1 to L2 2} we obtain
\[
\norm{\sigma^{\frac{1}{2}} \nabla g_h} \le C h^{-1}.
\]
Using~\eqref{eq:s32n_hs12n} we establish the lemma for $N=3$.
\end{proof}


\section{Weighted resolvent estimates}\label{sec: weighted resolvent}
In this section we will prove the weighted resolvent estimates in two and three dimensions. 
Since in this section (only) we will be dealing with complex valued function spaces, we need to  modify the definition of the $L^2$-inner product as
$$
(u,v)_\Om = \int_\Omega u(x) \bar v(x) \, dx,
$$
where $\bar v$ is the complex conjugate of $v$. Moreover we introduce the spaces $\mathbb{V} = H^1_0(\Omega) + i H^1_0(\Omega)$ and $\mathbb{V}_h = V_h + i V_h$.

In the continuous case for Lipschitz domains the following result was shown in \cite{ShenZW_1995a}: For any $\gamma\in (0,\frac{\pi}{2})$ there exists a constant $C = C_\gamma$ such that
\begin{equation}\label{eq: continuous resolvent}
\|(z+\Delta)^{-1}v\|_{L^p(\Om)}\le  \frac{C}{\abs{z}}\|v\|_{L^p(\Om)},\quad  z \in \C \setminus \Sigma_{\gamma}, \quad 1\le p\le \infty, \quad v \in L^p(\Omega),
\end{equation}
where $\Sigma_{\gamma}$ is defined by
\begin{equation}\label{eq: definition of sigma}
\Sigma_{\gamma}= \Set{z \in \mathbb{C} | \abs{\arg{z}} \le \gamma}.
\end{equation}
In the finite element setting, it is also known that
\begin{equation}\label{eq: resolven estimate}
\|(z+\Delta_h)^{-1}\chi\|_{L^p(\Om)}\le  \frac{C}{|z|}\|\chi\|_{L^p(\Om)},\quad\text{for all }\, z\in \mathbb{C}\setminus \Sigma_{\gamma},\;\ \chi\in \mathbb{V}_h
\end{equation}
for $1 \le p \le \infty$. For smooth domains such result is established in \cite{BakaevNY_ThomeeV_WahlbinLB_2003a} and for convex polyhedral domains in \cite{LeykekhmanD_VexlerB_2016c, LiB_SunW_2017a}.
In~\cite[Theorem 7]{LeykekhmanD_VexlerB_2016d} we also established the following weighted resolvent estimate:
\begin{equation}\label{eq:weighted_resolvent1}
\|\sigma^{\frac{N}{2}}(z+\Delta_h)^{-1}\chi\|_{L^2(\Om)}\le  \frac{C|\ln{h}|}{|z|}\|\sigma^{\frac{N}{2}}\chi\|_{L^2(\Om)},\quad\text{for all }\, z\in \mathbb{C}\setminus \Sigma_{\gamma},\; \chi\in \mathbb{V}_h.
\end{equation}

Our goal in this section is to establish another resolvent estimate in the weighted norm, which will be required later.

\begin{theorem}\label{thm: very weighted_Resolvent}
Let $N=2,3$. For any $\gamma\in(0,\frac{\pi}{2})$, there exists a constant $C$ independent of $h$ and $z$ such that
\[
\norm{\sigma^{\frac{N}{2}}\na\Delta^{-1}_h(z+\Delta_h)^{-1}\chi}_{L^2(\Om)} \le \frac{C \lh^{\frac{N-1}{N}}}{\abs{z}} \norm{\sigma^{\frac{N}{2}} \na\Delta^{-1}_h\chi}_{L^2(\Om)},\quad\text{for all }\, z\in \mathbb{C}\setminus \Sigma_{\gamma},\; \chi\in \mathbb{V}_h,
\]
 where $\Sigma_{\gamma}$ is defined in \eqref{eq: definition of sigma}.
\end{theorem}

Before we provide a proof of the above theorem we collect some preliminary results.

\subsection{Preliminary resolvent results}
The following lemma will be often used if dealing resolvent estimates.
\begin{lemma}\label{lemma:complex}
Let for each $z \in \mathbb{C}\setminus \Sigma_{\gamma}$ the numbers $\alpha_z, \beta_z \in \R_+$ be given and let $F_z = - z \alpha_z^2 + \beta_z^2$. Then there exists a constant $ C_\gamma$ such that
\[
\abs{z} \alpha_z^2 + \beta_z^2 \le C_\gamma \abs{F_z} \quad \text{for all }\, z\in \mathbb{C}\setminus \Sigma_{\gamma}.
\]
\end{lemma}
\begin{proof}
We consider the polar representation $-z \alpha^2_z = \abs{z} \alpha_z^2 e^{i \phi_z}$ with $\abs{\phi_z} \le \pi - \gamma$, since $\gamma \le \abs{\arg{z}}\le \pi$. This results in
\[
\abs{z} \alpha_z^2 e^{i \phi_z} + \beta_z^2 = F_z.
\]
Multiplying it by $e^{-i\phi_z/2}$ and taking real parts, we have
\[
\abs{z} \alpha_z^2 + \beta_z^2 \le (\cos{(\phi_z/2)})^{-1}|F_z| \le (\sin{(\gamma/2)})^{-1}|F_z| = C_\gamma|F_z|.
\]
\end{proof}

The following result is a best approximation type estimate in $H^1$ norm.
\begin{lemma}\label{lemma: resovent best approximation in H1}
Let $w \in \mathbb{V}$ and let $w_h \in \mathbb{V}_h$ with $e=w-w_h$ be defined by the orthogonality relation
\begin{equation}\label{eq: Galerkin orthogonality}
z(e,\chi)- (\na e,\na \chi) = 0, \quad \text{for all }\; \chi\in \mathbb{V}_h.
\end{equation}
Then there exists a constant $C>0$ such that for any $\chi\in\mathbb{V}_h$
\[
\norm{\na (w-w_h)}_{L^2(\Omega)}\le C\inf_{\chi\in\mathbb{V}_h}\left(h^{-1}\norm{w-\chi}_{L^2(\Omega)}+\norm{\na (w-\chi)}_{L^2(\Omega)}\right).
\]
\end{lemma}
\begin{proof}
Although the proof is straightforward, we will provide it for a completeness. Using \eqref{eq: Galerkin orthogonality}, for any $\chi\in \mathbb{V}_h$  we have
$$
-z\|e\|^2+\|\na e\|^2=-z(e,e)+(\na e,\na e)=-z(e,w-\chi)+(\na e,\na (w-\chi)):=F.
$$
Using the Cauchy-Schwarz inequality, we obtain
$$
|F|\le |z|\|e\|\|w-\chi\|+\|\na e\|\|\na(w-\chi)\|
$$
Hence, by Lemma~\ref{lemma:complex} and the Young's inequality, we have
\begin{align*}
|z|\|e\|^2+\|\na e\|^2&\le C_\gamma\left( |z|\|e\|\|w-\chi\|+\|\na e\|\|\na(w-\chi)\|\right)\\
&\le \frac{|z|}{2}\|e\|^2+\frac{C_\gamma^2}{2}|z|\|w-\chi\|^2+\frac{1}{2}\|\na e\|^2+\frac{C_\gamma^2}{2}\|\na(w-\chi)\|^2.
\end{align*}
Canceling, we obtain for all $z\in \mathbb{C}\setminus \Sigma_{\gamma}$
\begin{equation}\label{eq: basic energy estimate best approximation}
|z|\|e\|^2+\|\na e\|^2\le C^2_\gamma\left(|z|\|w-\chi\|^2+\|\na(w-\chi)\|^2\right).
\end{equation}
Now we consider two cases: $|z|\le  h^{-2}$ and $|z|>  h^{-2}$.

\textbf{Case 1: $|z|\le  h^{-2}$}.

Using that $|z|\le  h^{-2}$ from \eqref{eq: basic energy estimate best approximation} we immediately obtain
$$
\|\na e\|\le C_\gamma\left(h^{-1}\|w-\chi\|+\|\na(w-\chi)\|\right).
$$

\textbf{Case 2: $|z|>  h^{-2}$}.

In this case  from \eqref{eq: basic energy estimate best approximation}, we conclude
$$
\|e\|^2\le C_\gamma^2\left(\|w-\chi\|^2+\frac{1}{|z|}\|\na (w-\chi)\|^2\right)\le C_\gamma^2\left(\|w-\chi\|^2+h^2\|\na (w-\chi)\|^2\right).
$$
To estimate $\|\na e\|$ we use the triangle and the inverse estimate to obtain
$$
\begin{aligned}
\|\na e\|&\le \|\na (w-\chi)\|+\|\na (\chi-w_h)\|\\
&\le \|\na (w-\chi)\|+C_{\text{inv}}h^{-1}\|\chi-w_h\|\\
&\le \|\na (w-\chi)\|+C_{\text{inv}}h^{-1}(\|\chi-w\|+\|e\|)\\
&\le C_{\text{inv}}(1+C_\gamma)h^{-1}\|w-\chi\|+(C_{\text{inv}}C_\gamma+1)\|\na(w-\chi)\|.
\end{aligned}
$$
Combining both cases, we complete the proof.
\end{proof}

We will also need the following lemma.
\begin{lemma}\label{lemma: resolvent_l32_l3grad}
Let $w_h \in \mathbb{V}_h$ be the solution of
\[
z(w_h,\varphi)_{\Om}-(\na w_h,\na \varphi)_{\Om}=(f,\varphi)_{\Om}, \quad \text{for all }\; \varphi\in \mathbb{V}_h
\]
for some $f \in L^{\thalf}(\Omega) +  i L^{\thalf}(\Omega)$. There exists a constant $c>0$ such that
\[
\norm{\na w_h}_{L^3(\Omega)} \le C\norm{f}_{L^\thalf(\Omega)}.
\]
\end{lemma}
\begin{proof}
Let $w=(z+\Delta)^{-1}f$. From the resolvent estimates~\cite{ShenZW_1995a} we have
$$
\|(z+\Delta)^{-1}f\|_{L^\thalf(\Om)}\le \frac{C}{|z|}\|f\|_{L^\thalf(\Om)}\quad \text{and}\quad
\|\Delta(z+\Delta)^{-1}f\|_{L^\thalf(\Om)}\le C\|f\|_{L^\thalf(\Om)}.
$$
Therefore $\Delta w \in L^\thalf(\Omega)$ and using the elliptic regularity, see~\cite[Corollary 1]{FrommSJ_1993a}, we can conclude that $w \in W^{2,\thalf}(\Omega)$ with
\begin{equation}\label{eq:w_w232}
\|w\|_{W^{2,\frac{3}{2}}(\Om)}\le C\|f\|_{L^{\frac{3}{2}}(\Om)}.
\end{equation}
Since $W^{2,\frac{3}{2}}(\Om)$ is not embedded into $C(\Om)$, we use the Scott-Zhang interpolant $i^{SZ}_h$. Thus, by the triangle inequality we have
$$
\norm{\na w_h}_{L^3(\Om)} \le  \norm{\na w}_{L^3(\Omega)}+\norm{\na (w-i^{SZ}_hw)}_{L^3(\Omega)}+\norm{\na (w_h-i^{SZ}_hw)}_{L^3(\Omega)}:=J_1+J_2+J_3.
$$
Using the Sobolev embedding $W^{2,\frac{3}{2}}(\Om)\hookrightarrow W^{1,3}(\Om)$ and~\eqref{eq:w_w232} we have
$$
J_1\le \|w\|_{W^{2,\frac{3}{2}}(\Om)}\le C\|f\|_{L^{\frac{3}{2}}(\Om)}.
$$
Similarly, using stability of $i^{SZ}_h$ we have
$$
J_2\le \|w\|_{W^{2,\frac{3}{2}}(\Om)}\le C\|f\|_{L^{\frac{3}{2}}(\Om)}.
$$
To estimate $J_3$, we first use the inverse inequality, Lemma \ref{lemma: resovent best approximation in H1}, and \eqref{eq:I_h approximation SZ}, we have
$$
\begin{aligned}
J_3&\le Ch^{-\frac{1}{2}}\norm{\na (w_h-i^{SZ}_hw)}_{L^2(\Om)}\le Ch^{-\frac{1}{2}}\left(\norm{\na (w_h-w)}_{L^2(\Om)}+\norm{\na (w-i^{SZ}_hw)}_{L^2(\Om)}\right)\\
&\le Ch^{-\frac{1}{2}}\left(h^{-1}\norm{w-i^{SZ}_hw}_{L^2(\Om)}+\norm{\na (w-i^{SZ}_hw)}_{L^2(\Om)}\right)\le C\|w\|_{W^{2,\frac{3}{2}}(\Om)}\le C\|f\|_{L^{\frac{3}{2}}(\Om)}.
\end{aligned}
$$
Combining estimates for $J_1$, $J_2$, and $J_3$ we obtain the lemma.
\end{proof}

The following lemma is needed for the proof of our main resolvent estimate Theorem \ref{thm: very weighted_Resolvent}.
\begin{lemma}\label{lemma: needed for last step}
Let $N=2,3$. For a given $\chi \in \mathbb{V}_h$, let $u_h = (z+\Delta_h)^{-1} \chi$, or equivalently
\begin{equation}\label{eq: weak form zu_h}
z(u_h,\varphi)_{\Om}+(\Delta_h u_h,\varphi)_{\Om}=(\chi,\varphi)_{\Om}, \quad \text{for all }\; \varphi\in \mathbb{V}_h.
\end{equation}
Then for any $\gamma\in(0,\frac{\pi}{2})$, there exists a constant $C$ independent of $h$ and $z$ such that
\begin{equation}\label{eq: desired estimate for uh for N}
\|\si^{\frac{N-2}{2}}\Delta_h^{-1}u_h\|_{L^2(\Om)}\le \frac{C\lh^{\frac{N-1}{N}}}{|z|}\norm{\si^{\frac{N}{2}} \na\Delta_h^{-1} \chi}_{L^2(\Om)} \quad \text{for all }\, z\in \mathbb{C}\setminus \Sigma_{\gamma}.
\end{equation}
\end{lemma}
\begin{proof}
Most arguments will be using $L^2$ inner-products and $L^2$ norms over the whole domain $\Om$. To simplify the notation in this proof we will denote $\|\cdot\|_{L^2(\Om)}$ by $\|\cdot\|$ and $(\cdot,\cdot)_\Om$ by $(\cdot,\cdot)$.

We will consider the cases $N=2$ and $N=3$ separately.
Thus, for $N=2$, we need to show
\begin{equation}\label{eq: desired estimate for uh N=2}
\|\Delta_h^{-1}u_h\|\le \frac{C\lh^{\frac{1}{2}}}{|z|}\norm{\sigma \na\Delta_h^{-1} \chi}.
\end{equation}
To accomplish that, we test~\eqref{eq: weak form zu_h} with $\varphi=-\Delta_h^{-2}u_h$.
We obtain
$$
-z(u_h,\Delta_h^{-2}u_h)-(\Delta_h u_h,\Delta_h^{-2}u_h)=-(\chi,\Delta_h^{-2}u_h).
$$
Using that
$
(u_h,\Delta_h^{-2}  u_h)=\|\Delta_h^{-1} u_h\|^2
$
and
$
(\Delta_h u_h,\Delta_h^{-2}u_h)=-\|\na\Delta_h^{-1}u_h\|^2
$
we obtain
\begin{equation}\label{eq: estimates for z u_h with f N=2}
-z\|\Delta_h^{-1} u_h\|^2+\|\na \Delta_h^{-1} u_h\|^2=-(\chi,\Delta_h^{-2}u_h)=-(\Delta_h^{-1}\chi,\Delta_h^{-1}u_h).
\end{equation}
Using Lemma~\ref{lemma:complex} we obtain
\[
\abs{z}\|\Delta_h^{-1}u_h\|^2+\|\na \Delta_h^{-1} u_h\|^2\le C_{\gamma}\abs{(\Delta_h^{-1}\chi,\Delta_h^{-1}u_h)}, \quad\text{for}\quad z\in \C \setminus \Sigma_{\gamma}.
\]
For the right-hand side we have by the Cauchy-Schwarz and the Young's inequalities,
$$
\abs{(\Delta_h^{-1}\chi,\Delta_h^{-1}u_h)} \le \|\Delta_h^{-1}\chi\| \|\Delta_h^{-1}u_h\|\le \frac{|z|}{2C_\gamma}\|\Delta_h^{-1}u_h\|^2+\frac{C}{|z|}\|\Delta_h^{-1}\chi\|^2.
$$
 With the   Sobolev  $W^{1,1}(\Om)\hookrightarrow L^2(\Om)$ in two space dimensions, the Poincare inequality, and using the property of $\sigma$ \eqref{eq: property 1 of sigma}, we obtain
$$
\|\Delta_h^{-1}\chi\|\le C\|\Delta_h^{-1}\chi\|_{W^{1,1}(\Om)}\le  C\|\na\Delta_h^{-1}\chi\|_{L^1(\Om)}\le C\lh^{\frac{1}{2}}\|\si\na\Delta_h^{-1}\chi\|.
$$
Thus, we have
$$
\begin{aligned}
\abs{z}\|\Delta_h^{-1}u_h\|^2+\|\na \Delta_h^{-1} u_h\|^2\le \frac{C\lh}{|z|} \|\si\na\Delta_h^{-1}\chi\|^2+\frac{|z|}{2}\|\Delta_h^{-1}u_h\|^2.
\end{aligned}
$$
Kicking back $\frac{\abs{z}}{2}\|\Delta_h^{-1}u_h\|^2$, we establish \eqref{eq: desired estimate for uh N=2} and hence the lemma for $N=2$.

For $N=3$, we need to show
\begin{equation}\label{eq: desired estimate for uh N=3}
\|\si^{\frac{1}{2}}\Delta_h^{-1}u_h\|\le \frac{C\lh^{\frac{2}{3}}}{|z|}\norm{\si^{\frac{3}{2}} \na\Delta_h^{-1} \chi}.
\end{equation}
To accomplish that, we test~\eqref{eq: weak form zu_h} with $\varphi=-\Delta_h^{-1}P_h(\si\Delta_h^{-1}u_h)$.
We obtain
$$
-z(u_h,\Delta_h^{-1} P_h(\si\Delta_h^{-1}u_h))-(\Delta_h u_h,\Delta_h^{-1} P_h(\si \Delta_h^{-1}u_h))=-(\chi,\Delta_h^{-1} P_h(\si\Delta_h^{-1}u_h)).
$$
Using that
$$
(u_h,\Delta_h^{-1} P_h(\si\Delta_h^{-1} u_h))=(\Delta_h^{-1}u_h, P_h(\si\Delta_h^{-1} u_h))=\|\si^{\frac{1}{2}}\Delta_h^{-1} u_h\|^2
$$
and
\[
\begin{aligned}
(\Delta_h u_h,\Delta_h^{-1} P_h(\si \Delta_h^{-1}u_h))&=(\Delta_h \Delta_h^{-1}u_h, P_h(\si \Delta_h^{-1}u_h))\\
&=-(\na \Delta_h^{-1}u_h,\na P_h(\si \Delta_h^{-1}u_h))\\
&=-(\na \Delta_h^{-1}u_h,\na (\si \Delta_h^{-1}u_h))-(\na \Delta_h^{-1}u_h,\na( P_h-\operatorname{Id})(\si \Delta_h^{-1}u_h))\\
&=-\|\si^{\frac{1}{2}}\na \Delta_h^{-1} u_h\|^2
-(\na \Delta_h^{-1}u_h,\na \si \Delta_h^{-1}u_h)\\
&\qquad \qquad \qquad \qquad -(\si^{\frac{1}{2}}\na \Delta_h^{-1}u_h,\si^{-\frac{1}{2}}\na( P_h-\operatorname{Id})(\si \Delta_h^{-1}u_h)),
\end{aligned}
\]
we obtain
\begin{equation}\label{eq: estimates for z u_h with f N=3}
-z\|\sigma^{\frac{1}{2}} \Delta_h^{-1} u_h\|^2+\|\sigma^{\frac{1}{2}} \na \Delta_h^{-1} u_h\|^2=F,
\end{equation}
where
\[
\begin{aligned}
F=F_1+F_2+F_3&:=-(\chi,\Delta_h^{-1} P_h(\si\Delta_h^{-1}u_h))-(\na \Delta_h^{-1}u_h,\na \si \Delta_h^{-1}u_h)\\&-(\si^{\frac{1}{2}}\na \Delta_h^{-1}u_h,\si^{-\frac{1}{2}}\na( P_h-\operatorname{Id})(\si \Delta_h^{-1}u_h)).
\end{aligned}
\]
Using Lemma~\ref{lemma:complex} we conclude
\begin{equation}\label{eq: estimate of the second step with f}
\abs{z}\|\si^{\frac{1}{2}} \Delta_h^{-1} u_h\|^2+\|\si^{\frac{1}{2}}\na \Delta_h^{-1} u_h\|^2\le C_\gamma|F|, \quad\text{for}\quad z\in \mathbb{C}\setminus \Sigma_{\gamma}.
\end{equation}
By the Cauchy-Schwarz and the Young's inequalities,
\begin{align*}
|F_1|\le \|\si^{\frac{1}{2}}\Delta_h^{-1}\chi\| \|\si^{\frac{1}{2}}\Delta_h^{-1}u_h\|&\le \frac{CC_\gamma}{|z|}\|\si^{\frac{1}{2}}\Delta_h^{-1}\chi\|^2+ \frac{|z|}{4C_\gamma}\|\si^{\frac{1}{2}}\Delta_h^{-1}u_h\|^2\\
&\le \frac{CC_\gamma}{|z|}\|\si^{\frac{3}{2}}\na\Delta_h^{-1}\chi\|^2+ \frac{|z|}{4C_\gamma}\|\si^{\frac{1}{2}}\Delta_h^{-1}u_h\|^2,
\end{align*}
where in the last step we again use Lemma \ref{lemma: from weighted L2 to weighted in H1} with $\alpha=\frac{1}{2}$, $\beta=0$, and $p=2$.
To estimate $F_2$ we use  the Cauchy-Schwarz and the Young's inequalities, as well as the fact that $\abs{\nabla \si}\le C$.
$$
|F_2|\le C\|\si^{\frac{1}{2}}\na\Delta_h^{-1}u_h\| \|\si^{-\frac{1}{2}}\Delta_h^{-1}u_h\|\le \frac{1}{4C_\gamma}\|\si^{\frac{1}{2}}\na\Delta_h^{-1}u_h\|^2+CC_\gamma\|\si^{-\frac{1}{2}}\Delta_h^{-1}u_h\|^2.
$$
Using Lemma \ref{lemma: from weighted L2 to weighted in H1} with $\alpha=\beta=-\frac{1}{2}$, $p=\frac{3}{2}$ and $p'=3$, we obtain
$$
\|\si^{-\frac{1}{2}}\Delta_h^{-1}u_h\|^2\le C\|\Delta_h^{-1}u_h\|_{L^{\frac{3}{2}}(\Om)}\|\na\Delta_h^{-1}u_h\|_{L^3(\Om)}.
$$
Using the properties of $\si$ and the H\"{o}lder inequality, we have
$$
\|\Delta_h^{-1}u_h\|_{L^{\frac{3}{2}}(\Om)}\le C\lh^{\frac{1}{6}}\|\si^{\frac{1}{2}}\Delta_h^{-1}u_h\|,
$$
and as a result
\begin{equation}\label{eq: estimate for f_2 after all steps}
 |F_2|\le \frac{1}{4C_\gamma}\|\si^{\frac{1}{2}}\na\Delta_h^{-1}u_h\|^2+ \frac{|z|}{4C_\gamma}\|\si^{\frac{1}{2}}\Delta_h^{-1}u_h\|^2+\frac{C}{|z|}\lh^{\frac{1}{3}}\|\na\Delta_h^{-1}u_h\|^2_{L^3(\Om)}.
\end{equation}
Finally, using the Cauchy-Schwarz inequality, Lemma \ref{lemma:super_ih_ph},  and the Young's inequality, we obtain
$$
|F_3|\le C\|\si^{\frac{1}{2}}\na \Delta_h^{-1}u_h\|\|\si^{-\frac{1}{2}}\Delta_h^{-1}u_h\|\le \frac{1}{4C_\gamma}\|\si^{\frac{1}{2}}\na \Delta_h^{-1}u_h\|^2+C_\gamma\|\si^{-\frac{1}{2}}\Delta_h^{-1}u_h\|^2.
$$
Similarly to the estimate of $F_2$ above we obtain,
\begin{equation}\label{eq: estimate for f_3 after all steps}
 |F_3|\le \frac{1}{4C_\gamma}\|\si^{\frac{1}{2}}\na\Delta_h^{-1}u_h\|^2+ \frac{|z|}{4C_\gamma}\|\si^{\frac{1}{2}}\Delta_h^{-1}u_h\|^2+\frac{C}{|z|}\lh^{\frac{1}{3}}\|\na\Delta_h^{-1}u_h\|^2_{L^3(\Om)}.
\end{equation}
Combining estimates for $F_1$, $F_2$, and $F_3$, inserting them into \eqref{eq: estimate of the second step with f} and kicking back, we obtain
\begin{equation}\label{eq: sigma3/2 Gh after second step N=3}
\abs{z}\|\si^{\frac{1}{2}} \Delta_h^{-1} u_h\|^2+\|\si^{\frac{1}{2}}\na\Delta^{-1} u_h\|^2\le \frac{C}{|z|}\|\si^{\frac{3}{2}}\na\Delta_h^{-1}\chi\|^2 +\frac{C}{|z|}\lh^{\frac{1}{3}}\|\na\Delta_h^{-1}u_h\|^2_{L^3(\Om)}.
\end{equation}
Thus, in order to establish the lemma for $N=3$, we need to show
\begin{equation}\label{eq: desired second estimate for uh N=3}
\|\na\Delta_h^{-1}u_h\|_{L^3(\Om)}\le C\lh^{\frac{1}{2}}\|\si^{\frac{3}{2}}\na\Delta_h^{-1}\chi\|.
\end{equation}
This estimates follows by Lemma \ref{lemma: resolvent_l32_l3grad}, Sobolev embedding theorem $W^{1,1}(\Omega) \hookrightarrow L^{\frac{3}{2}}(\Omega)$ combined with the Poincare inequality, and the properties of $\si$. Indeed,
$$
\|\na\Delta_h^{-1}u_h\|_{L^3(\Om)}\le C\|\Delta_h^{-1}\chi\|_{L^{\frac{3}{2}}(\Om)}\le C\|\na\Delta_h^{-1}\chi\|_{L^1(\Om)}\le C\lh^{\frac{1}{2}}\|\si^{\frac{3}{2}}\na\Delta_h^{-1}\chi\|.
$$
This concludes the proof of the lemma.
\end{proof}

\subsection{Proof of Theorem \ref{thm: very weighted_Resolvent}}

For an arbitrary $\chi \in \mathbb{V}_h$, the solution to resolvent equation $u_h$ satisfies
\begin{equation}\label{eq: weak form zu_h N=3}
z(u_h,\varphi)+(\Delta_h u_h,\varphi)=(\chi,\varphi), \quad \text{for all }\; \varphi\in \mathbb{V}_h.
\end{equation}
First we test \eqref{eq: weak form zu_h N=3}
with $\varphi=\Delta_h^{-1} P_h(\sigma^Nu_h)$ to obtain
$$
z(u_h,\Delta_h^{-1} P_h(\sigma^Nu_h))+(\Delta_h u_h,\Delta_h^{-1} P_h(\sigma^Nu_h))=(\chi,\Delta_h^{-1} P_h(\sigma^Nu_h)).
$$
Using that
$$
(\Delta_h u_h,\Delta_h^{-1} P_h(\sigma^Nu_h))=(u_h, P_h(\sigma^Nu_h))=(u_h,\sigma^Nu_h)=\|\si^{\frac{N}{2}} u_h\|^2
$$
and
\begin{align*}
(u_h,&\Delta_h^{-1} P_h(\sigma^Nu_h))=(\Delta_h^{-1}u_h, P_h(\sigma^Nu_h))=
(\Delta_h^{-1}u_h, \si^Nu_h)=(\si^N\Delta_h^{-1}u_h, \Delta_h\Delta_h^{-1}u_h)\\
&=-(\na (P_h\sigma^N\Delta_h^{-1}u_h), \na\Delta_h^{-1}u_h)\\
&=-(\na (\sigma^N\Delta_h^{-1}u_h), \na\Delta_h^{-1}u_h)-(\na (P_h-\operatorname{Id})(\si^N\Delta_h^{-1}u_h), \na\Delta_h^{-1}u_h)\\
&=-\|\sigma^{\frac{N}{2}}\na\Delta_h^{-1}u_h\|^2-N(\si^{N-1}\na \sigma\Delta_h^{-1}u_h, \na\Delta_h^{-1}u_h)-(\na (P_h-\operatorname{Id})(\sigma^N\Delta_h^{-1}u_h), \na\Delta_h^{-1}u_h),
\end{align*}
we obtain
\begin{equation}\label{eq: estimates for z u_h with F N=3}
-z\|\sigma^{\frac{N}{2}} \na\Delta_h^{-1} u_h\|^2+\|\sigma^{\frac{N}{2}}  u_h\|^2=F,
\end{equation}
where
$$
\begin{aligned}
F&=F_1+F_2+F_3\\
&:=(\chi,\Delta_h^{-1} P_h(\si^Nu_h))+Nz(\si^{N-1}\na \si\Delta_h^{-1}u_h, \na\Delta_h^{-1}u_h)
+z(\si^{-\frac{N}{2}}\na (P_h-\operatorname{Id})(\si^N\Delta_h^{-1}u_h), \si^{\frac{N}{2}}\na\Delta_h^{-1}u_h).
\end{aligned}
$$
By Lemma~\ref{lemma:complex} we conclude
$$
\abs{z}\|\si^{\frac{N}{2}} \na\Delta_h^{-1} u_h\|^2+\|\si^{\frac{N}{2}} u_h\|^2\le C_\gamma|F|, \quad\text{for}\quad z\in \mathbb{C} \setminus \Sigma_{\gamma}.
$$
To estimate $F_1$ we notice that
\[
\begin{aligned}
(\chi,\Delta_h^{-1} P_h(\si^Nu_h))&=(\Delta_h^{-1}\chi, P_h(\si^Nu_h))=(\si^N\Delta_h^{-1}\chi, u_h)\\
&=(P_h(\si^N\Delta_h^{-1}\chi),\Delta_h\Delta_h^{-1} u_h)\\
&=-(\na P_h(\sigma^N\Delta_h^{-1}\chi),\na\Delta_h^{-1} u_h)\\
&=-(\na (\sigma^N\Delta_h^{-1}\chi),\na\Delta_h^{-1} u_h)-(\na (P_h-\operatorname{Id})(\sigma^N\Delta_h^{-1}\chi),\na\Delta_h^{-1} u_h)\\
&=-(\sigma^N\na\Delta_h^{-1}\chi,\na\Delta_h^{-1} u_h)-N(\si^{N-1}\na\si\Delta_h^{-1}\chi,\na\Delta_h^{-1} u_h)\\
&\qquad\qquad\qquad\qquad\qquad -(\si^{-\frac{N}{2}}\na (P_h-\operatorname{Id})(\si^N\Delta_h^{-1}\chi),\si^{\frac{N}{2}}\na\Delta_h^{-1} u_h).
\end{aligned}
\]
Using $\abs{\nabla \sigma} \le C$, the Cauchy-Schwarz inequality, and the Young's inequality, we obtain,
\[
\begin{aligned}
|F_1| &\le  \|\si^{\frac{N}{2}} \na\Delta_h^{-1}u_h\|\|\si^{\frac{N}{2}}\na\Delta_h^{-1}\chi\|+C\|\si^{\frac{N-2}{2}}\Delta_h^{-1}\chi\|\|\si^{\frac{N}{2}}\na\Delta_h^{-1} u_h\|\\
&\qquad\qquad\qquad\qquad\qquad+\|\si^{-\frac{N}{2}}\na (P_h-\operatorname{Id})(\sigma^N\Delta_h^{-1}\chi)\|\|\si^{\frac{N}{2}}\na\Delta_h^{-1} u_h\|\\
&\le \frac{CC_{\gamma}}{\abs{z}}\left(\|\si^{\frac{N}{2}}\na\Delta_h^{-1}\chi\|^2+\|\si^{\frac{N-2}{2}}\Delta_h^{-1}\chi\|^2\right)
+\frac{\abs{z}}{4C_\gamma}\|\si^{\frac{N}{2}} \na\Delta_h^{-1} u_h\|^2,
\end{aligned}
\]
where in the last step we used Lemma \ref{lemma:super_ih_ph} to obtain
$$
\|\si^{-\frac{N}{2}}\na (P_h-\operatorname{Id})(\sigma^N\Delta_h^{-1}\chi)\|\le C\|\si^{\frac{N-2}{2}}\Delta_h^{-1}\chi\|.
$$
For $N=2$, using the   Sobolev  embedding $W^{1,1}(\Omega)\hookrightarrow L^2(\Omega)$ and the Poincare inequality, we obtain
$$
\|v_h\|\le C\|v_h\|_{W^{1,1}(\Om)}\le C\|\na v_h\|_{L^1(\Om)}, \quad \text{for all }\; v_h\in V_h.
$$
Using in addition the property of $\sigma$ \eqref{eq: property 1 of sigma}, we obtain
$$
\|\Delta_h^{-1}\chi\|\le C\|\na \Delta_h^{-1}\chi\|_{L^1(\Om)}\le C\abs{\ln{h}}^{\frac{1}{2}}\|\si\na \Delta_h^{-1}\chi\|.
$$
For $N=3$, we use Lemma \ref{lemma: from weighted L2 to weighted in H1} with $\alpha=\frac{1}{2}$, $\beta=0$, and $p=2$, to obtain
$$
\|\si^{\frac{1}{2}}\Delta_h^{-1}\chi\|\le C\|\si^{\frac{3}{2}}\na \Delta_h^{-1}\chi\|.
$$
Thus,
$$
|F_1|\le  \frac{C C_\gamma\lh^{3-N}}{\abs{z}}\|\si^{\frac{N}{2}}\na\Delta_h^{-1}\chi\|^2+\frac{\abs{z}}{4C_\gamma}\|\si^{\frac{N}{2}} \na\Delta_h^{-1} u_h\|^2.
$$
To estimate $F_2$ we use  the Cauchy-Schwarz and the Young's inequalities,
$$
|F_2|\le C|z|\|\si^{\frac{N-2}{2}}\Delta_h^{-1}u_h\| \|\si^{\frac{N}{2}}\na\Delta_h^{-1}u_h\|\le \frac{\abs{z}}{4C_\gamma}\|\si^{\frac{N}{2}} \na\Delta_h^{-1} u_h\|^2+C C_{\gamma}\abs{z}\|\si^{\frac{N-2}{2}}\Delta_h^{-1}u_h\|^2.
$$
To estimate $F_3$ we use Lemma \ref{lemma:super_ih_ph}, the Cauchy-Schwarz and the Young's inequalities,
$$
|F_3|\le C |z|\|\si^{-\frac{N}{2}}\na (P_h-\operatorname{Id})(\sigma^N\Delta_h^{-1}u_h)\|\|\si^{\frac{N}{2}}\na\Delta_h^{-1}u_h\|\le C_{\gamma}\abs{z}\|\si^{\frac{N-2}{2}}\Delta_h^{-1}u_h\|^2 + \frac{\abs{z}}{4C_\gamma}\|\si^{\frac{N}{2}} \na\Delta_h^{-1} u_h\|^2.
$$
Combining estimates for $F_1$, $F_2$, $F_3$ and kicking back, we obtain
\begin{equation}\label{eq: sigma3/2 Gh after first step N=3}
\abs{z}\|\si^{\frac{N}{2}} \na\Delta_h^{-1} u_h\|^2+\|\si^{\frac{N}{2}}  u_h\|^2\le \frac{C \lh^{3-N}}{\abs{z}}\norm{\si^{\frac{N}{2}} \na\Delta_h^{-1}\chi}^2+C\abs{z}\|\si^{\frac{N-2}{2}}\Delta_h^{-1}u_h\|^2.
\end{equation}
Now applying Lemma \ref{lemma: needed for last step} to the last term concludes the proof of the theorem.

\section{Discrete maximal parabolic estimates}\label{sec: discrete max}

In this section we state stability results for inhomogeneous problems that are central in establishing our main results.
Since we apply the following results for different norms on $V_h$, namely, for $L^p(\Om)$, weighted $L^2(\Om)$, and weighted $H^{-1}(\Omega)$ norms, we state them for a general Banach norm $\vertiii{\cdot}$.

Let  $\vertiii{\cdot}$ be a norm on $V_h$ (naturally extended to a norm on $\mathbb{V}_h$) such that for some $\gamma\in(0,\frac{\pi}{2})$ the following resolvent estimate holds,
\begin{equation}\label{eq: resolvent in Banach space}
\vertiii{(z+\Delta_h)^{-1}\chi} \le \frac{M_h}{\abs{z}} \vertiii{\chi}, \quad \text{for all }\; z \in \mathbb{C} \setminus \Sigma_\gamma, \quad \chi \in \mathbb{V}_h,
\end{equation}
where $\Sigma_{\gamma}$ is defined in \eqref{eq: definition of sigma} and
 the constant $M_h$ is independent of $z$.

This assumption is fulfilled for $\vertiii{\cdot}=\norm{\cdot}_{L^p(\Omega)}$, $1\le p \le \infty$, with a constant $M_h \le C$ independent of $h$, see~\cite{LiB_SunW_2017a}, for $\vertiii{\cdot}=\norm{\sigma^{\frac{N}{2}}(\cdot)}_{L^2(\Omega)}$ with $M_h \le C \lh$, see~\cite[Theorem 7]{LeykekhmanD_VexlerB_2016d}, and for $\vertiii{\cdot}=\norm{\sigma^{\frac{N}{2}}\na\Delta_h^{-1}(\cdot)}_{L^2(\Omega)}$ with $M_h \le C \lh^{\frac{N-1}{N}}$, see Theorem~\ref{thm: very weighted_Resolvent}.

We consider the inhomogeneous heat equation \eqref{eq: heat equation},  with $u_0=0$ and its discrete approximation $u_{kh} \in \Xkh$ defined by
\begin{equation}\label{eq: dGr nonhomogeneous equation fully}
B(u_{kh},\varphi_{kh})=(f,\varphi_{kh}),\quad \text{for all }\; \varphi_{kh}\in \Xkh.
\end{equation}
The next result is a discrete maximal parabolic regularity result \cite[Theorem~14]{LeykekhmanD_VexlerB_2016b}.
\begin{lemma}[Discrete maximal parabolic regularity]\label{lemma: fully discrete_maximal_parabolic}
Let $\vertiii{\cdot}$ be a norm on $V_h$ fulfilling the resolvent estimate~\eqref{eq: resolvent in Banach space} and let $1 \le s \le \infty$. Let $u_{kh}$ be a solution  of \eqref{eq: dGr nonhomogeneous equation fully}. Then, there exists a constant $C$ independent of $k$ and $h$ such that
\begin{equation*}
\begin{aligned}
\left(\sum_{m=1}^M\int_{I_m}\vertiii{\pa_t u_{kh}(t)}^sdt\right)^{\frac{1}{s}}+\left(\sum_{m=1}^M\int_{I_m}\vertiii{\Delta_h u_{kh}(t)}^sdt\right)^{\frac{1}{s}}&+\left(\sum_{m=1}^Mk_m\vertiii{k_m^{-1}[u_{kh}]_{m-1}}^s\right)^{\frac{1}{s}}\\
&\le C M_h \lk\left(\int_I\vertiii{P_h f(t)}^sdt\right)^{\frac{1}{s}},
\end{aligned}
\end{equation*}
with obvious change of notation in the case $s=\infty$. For $m=1$ the jump is understood as $[u_{kh}]_0=u_{kh,0}^+$.
\end{lemma}
%

\section{Proofs of pointwise global best approximation results}\label{sec: proofs global results}

We are now ready to establish our main results.

\subsection{Proof of Theorem \ref{thm:global_best_approx}}
\begin{proof}
Let $\tilde{t} \in(0,T]$ and let $x_0 \in \Omega$ be an arbitrary but fixed point. Without loss of generality we assume $\tilde{t} \in I_M=(t_{M-1},T]$. Note, that the case $\tilde t = 0$ is trivial, since $u_{kh}(0) = P_h u_0$ and the statement of the theorem follows by  the stability of the $L^2$ projection in the $W^{1,\infty}(\Omega)$ norm. This stability result is a consequence of the stability in the $L^\infty(\Omega)$ norm, see~\cite{DouglasDupontWahlbin:1975} and the standard inverse inequality.

We consider the following regularized Green's function
\begin{equation}\label{eq: heat with dirac on the RHS}
\begin{aligned}
-\tilde{g}_t(t,x)-\Delta \tilde{g}(t,x) &= D\tilde{\delta}_{x_0}(x)\tilde{\theta}(t)& (t,x)\in \IOm,\;  \\
\tilde{g}(t,x)&=0, &(t,x) \in I\times\pa\Omega, \\
\tilde{g}(T,x)&=0, & x \in \Omega,
\end{aligned}
\end{equation}
where $\tilde{\delta}_{x_0}$ is the smoothed Dirac introduced in \eqref{eq: definition delta}, $D$ denotes an arbitrary partial derivative in space, and $\tilde{\theta}\in C^\infty(0,T)$ is the regularized Delta function in time with properties $\supp(\tilde{\theta})\subset I_M$, $\|\tilde{\theta}\|_{L^1(I_M)}\le C$ and
$$
(\tilde{\theta},\varphi_k)_{I_M}=\varphi_k(\tilde{t}),\quad \text{for all }\; \varphi_k\in X^q_k.
$$
Let $\tilde{g}_{kh}$ be dG($q$)cG($r$) approximation of $\tilde{g}$, i.e. $B(\varphi_{kh},\tilde{g}-\tilde{g}_{kh})=0$.
Then we have
$$
\begin{aligned}
-Du_{kh}(\tilde{t},x_0)&=(u_{kh}, D\tilde{\delta}_{x_0}\tilde{\theta})=B(u_{kh},\tilde{g})=B(u_{kh},\tilde{g}_{kh})=B(u,\tilde{g}_{kh})\\
&=-\sum_{m=1}^{M}(u,\pa_t\tilde{g}_{kh})_{\ImOm}+(\na  u, \na \tilde{g}_{kh})_{\IOm}-\sum_{m=1}^{M}(u_m,[\tilde{g}_{kh}]_m)_{\Om}=J_1+J_2+J_3,
\end{aligned}
$$
where  in the sum with jumps we included the last term by setting $\tilde{g}_{kh,M+1} = 0$ and defining
consequently $[\tilde{g}_{kh}]_M = -\tilde{g}_{kh,M}$.
Using the H\"{o}lder inequality, stability of the Ritz projection in $W^{1,\infty}(\Omega)$ from~\cite{JGuzman_DLeykekhman_JRossmann_AHSchatz_2009a} and the $L^\infty$ error estimate from Lemma~\ref{lemma: approximation of Ritz} we have
\[
\begin{aligned}
    J_1&=-\sum_{m=1}^{M}\bigg((R_hu,\Delta_h\Delta_h^{-1}\pa_t\tilde{g}_{kh})_{\ImOm}+((I-R_h)u,\pa_t\tilde{g}_{kh})_{\ImOm}\bigg)\\
&=\sum_{m=1}^{M}\bigg((\na R_hu,\na\Delta_h^{-1}\pa_t\tilde{g}_{kh})_{\ImOm}-((I-R_h)u,\pa_t\tilde{g}_{kh})_{\ImOm}\bigg)\\
&\le \sum_{m=1}^{M}\bigg(\|\na u\|_{L^\infty(\ImOm)}\|\na\Delta_h^{-1}\pa_t\tilde{g}_{kh}\|_{L^1(I_m; L^1(\Om))}+\|(I-R_h)u\|_{L^\infty(\ImOm)}\|\pa_t\tilde{g}_{kh}\|_{L^1(I_m; L^1(\Om))}\bigg)\\
&\le C\lh^{\frac{1}{2}}\|\na u\|_{L^\infty(I\times\Om)}\sum_{m=1}^{M}\|\si^{\frac{N}{2}}\na\Delta_h^{-1}\pa_t\tilde{g}_{kh}\|_{L^1(I_m; L^2(\Om))}\\
&\ +Ch\lh \|\na u\|_{L^\infty(I\times\Om)}\sum_{m=1}^{M}\|\pa_t\tilde{g}_{kh}\|_{L^1(I_m; L^1(\Om))}.
\end{aligned}
\]
Applying the discrete maximal parabolic regularity result from Lemma~\ref{lemma: fully discrete_maximal_parabolic} with respect to $\norm{\sigma^{\frac{N}{2}}\na\Delta_h^{-1}(\cdot)}_{L^2(\Omega)}$ and with respect to the $L^1(\Omega)$ norm we get
\begin{equation}\label{eq: estimate for J1 pointwise time}
\begin{aligned}
J_1 & \le C\lk\|\na u\|_{L^\infty(I\times\Om)}\left(\lh^{\half+\frac{N-1}{N}}\|\si^{\frac{N}{2}}\na\Delta_h^{-1} P_hD\tilde{\delta}\|_{L^2(\Om)}\|\tilde{\theta}\|_{L^1(I_M)}
+h\lh\| P_h D\tilde{\delta}\|_{L^1(\Om)}\|\tilde{\theta}\|_{L^1(I_M)}\right)\\
& \le C\lh^{\frac{2N-1}{N}}\lk\|\na u\|_{L^\infty(I\times\Om)},
\end{aligned}
\end{equation}
where in the last step we used Lemma~\ref{lemma: Delta_h-1 for Di delta}, Lemma~\ref{lemma:sigma_delta} and the fact that $\|\tilde{\theta}\|_{L^1(I_M)} \le C$.
Similarly, using the H\"{o}lder inequality,  properties of $\sigma$, Lemma \ref{lemma: fully discrete_maximal_parabolic}, and Lemma \ref{lemma: Delta_h-1 for Di delta}, we have
\begin{equation}\label{eq: estimate for J2 pointwise time}
\begin{aligned}
    J_2= (\na  u, \na g_{kh})_{\IOm}
&\le \|\na u\|_{L^\infty(\IOm))}\|\na \tilde{g}_{kh}\|_{L^1(I; L^1(\Om))}\\
&\le C\lh^{\frac{1}{2}}\|\na u\|_{L^\infty(I\times\Om)}\|\si^{\frac{N}{2}}\na \tilde{g}_{kh}\|_{L^1(I; L^2(\Om))}\\
&\le C\lh^{\frac{1}{2}+\frac{N-1}{N}}\lk\|\na u\|_{L^\infty(I\times\Om)}\|\si^{\frac{N}{2}}\na\Delta_h^{-1}P_h D\tilde{\delta}\|_{L^2(\Om)}\|\tilde{\theta}\|_{L^1(I_M)}\\
&\le C\lh^{\frac{2N-1}{N}}\lk\|\na u\|_{L^\infty(I\times\Om)}.
\end{aligned}
\end{equation}
Similarly to the estimate of $J_1$, using the H\"{o}lder inequality, properties of $\sigma$, and Lemma~\ref{lemma: approximation of Ritz} we have
\[
\begin{aligned}
    J_3&=-\sum_{m=1}^{M}\bigg((R_h u_m,[\tilde{g}_{kh}]_m)_{\Om}+((I-R_h) u_m,[\tilde{g}_{kh}]_m)_{\Om}\bigg)\\
&=\sum_{m=1}^{M}\bigg((\na u_m,[\na\Delta_h^{-1} \tilde{g}_{kh}]_m)_{\Om}-((I-R_h) u_m,[\tilde{g}_{kh}]_m)_{\Om}\bigg)\\
&\le \sum_{m=1}^{M}\|\na u_m\|_{L^\infty(\Om)}\|[\na\Delta_h^{-1} \tilde{g}_{kh}]_m\|_{L^1(\Om)}+\sum_{m=1}^{M}\|(I-R_h) u_m\|_{L^\infty(\Om)}\|[ \tilde{g}_{kh}]_m\|_{L^1(\Om)}\\
&\le C\lh^{\frac{1}{2}}\|\na u\|_{L^\infty(I\times\Om)}\sum_{m=1}^{M}\|\si^{\frac{N}{2}}[\na\Delta_h^{-1} \tilde{g}_{kh}]_m\|_{L^2(\Om)}+Ch\lh\|\na u\|_{L^\infty(I \times \Om)}\sum_{m=1}^{M}\|[ \tilde{g}_{kh}]_m\|_{L^1(\Om)}.
\end{aligned}
\]
Applying the discrete maximal parabolic regularity result from Lemma~\ref{lemma: fully discrete_maximal_parabolic} with respect to $\norm{\sigma^{\frac{N}{2}}\na\Delta_h^{-1}(\cdot)}_{L^2(\Omega)}$ and with respect to the $L^1(\Omega)$ norm we get
\begin{equation}\label{eq: estimate for J3 pointwise time}
\begin{aligned}
J_3 &\le C\lk\|\na u\|_{L^\infty(I\times\Om)}\left(\lh^{\half+\frac{N-1}{N}}\|\si^{\frac{N}{2}}\na\Delta_h^{-1} P_h D\tilde{\delta}\|_{L^2(\Om)}\|\tilde{\theta}\|_{L^1(I_M)}+ h\lh\|P_h D\tilde{\delta}\|_{L^1(\Om)}\|\tilde{\theta}\|_{L^1(I_M)}\right)\\
&\le C \lh^{\frac{2N-1}{N}}\lk\|\na u\|_{L^\infty(I\times\Om)},
\end{aligned}
\end{equation}
where in the last step we again used Lemma~\ref{lemma: Delta_h-1 for Di delta}, Lemma~\ref{lemma:sigma_delta}, and the fact that $\|\tilde{\theta}\|_{L^1(I_M)} \le C$.
Combining the estimates for $J_1$, $J_2$, and $J_3$, and taking supremum over all partial derivatives, we conclude that
$$
|\na u_{kh}(\tilde{t},x_0)|\le  C\ell_h\ell_k\|\na u\|_{L^\infty(I \times \Om)}.
$$
Using  that the dG($q$)cG($r$) method is invariant on $\Xkh$, by replacing $u$ and $u_{kh}$ with $u-\chi$ and $u_{kh}-\chi$ for any $\chi\in \Xkh$, and using the triangle inequality we obtain Theorem~\ref{thm:global_best_approx}.
\end{proof}

\section{Proof of pointwise interior best approximation results}\label{sec: proofs local results}

\subsection{Proof of Theorem \ref{thm:local_best_approx}}

To obtain the interior estimate we introduce a smooth cut-off function $\omega$ with the properties that
\begin{subequations} \label{def: properties of omega a local}
\begin{align}
\omega(x)&\equiv 1,\quad x\in B_d \label{eq: property 1 of omega_a local}\\
\omega(x)&\equiv 0,\quad x\in \Om\backslash B_{2d} \label{eq: property 2 of omega_a local}\\
 |\na \omega|&\le Cd^{-1}, \quad |\na^2 \omega|\le Cd^{-2}, \label{eq: property 3 of omega_a local}
\end{align}
\end{subequations}
where $B_d=B_d(x_0)$ is a ball of radius $d$ centered at $x_0$.

As in the proof of Theorem~\ref{thm:global_best_approx}, we obtain
\begin{equation}\label{eq: local starting expression}
-Du_{kh}(\tilde t,x_0)=B(u_{kh}, \tilde{g}_{kh}) =B(u, \tilde{g}_{kh})=B(\om u, \tilde{g}_{kh})+B((1-\om)u, \tilde{g}_{kh}),
\end{equation}
where $\tilde{g}_{kh}$ is the solution of~\eqref{eq: heat with dirac on the RHS}. The first term can be estimated using the global result from Theorem~\ref{thm:global_best_approx}. To this end we introduce $\tilde u=\omega u$ and the corresponding dG($q$)cG($r$) solution $\tilde u_{kh} \in \Xkh$ defined by
\[
B(\tilde u_{kh} - \tilde u, \varphi_{kh}) = 0 \quad \text{for all }  \varphi_{kh} \in \Xkh.
\]
There holds
\[
\begin{aligned}
B(\tilde u, \tilde{g}_{kh}) = B(\tilde u_{kh}, g_{kh}) = -D\tilde u_{kh}(\tilde t,x_0)
&\le   C \ell_k \ell_h  \norm{\na \tilde{u}}_{L^\infty(I\times \Om)}\\
&\le C \ell_k \ell_h  \left(d^{-1}\norm{u}_{L^\infty(I\times B_{2d})}+\norm{\na u}_{L^\infty(I\times B_{2d})}\right).
\end{aligned}
\]
This results in
\begin{equation}\label{eq:local_after_first_step}
|\na u_{kh}(\tilde{t},x_0)| \le C \ell_k \ell_h  \left(d^{-1}\norm{u}_{L^\infty(I\times B_{2d})}+\norm{\na u}_{L^\infty(I\times B_{2d})}\right) + B((1-\om)u, \tilde{g}_{kh}).
\end{equation}
It remains to estimate the term $B((1-\om)u, \tilde{g}_{kh})$. Using the dual expression~\eqref{eq:B_dual} of the bilinear form $B$ we obtain
\begin{equation}\label{eq:J1_J2_local}
\begin{aligned}
B((1-\om)u, \tilde{g}_{kh}) &=-\sum_{m=1}^{M}((1-\om)u,\pa_t\tilde{g}_{kh})_{\ImOm}+(\na((1-\om)u), \na \tilde{g}_{kh})_{\IOm}
\\&- \sum_{m=1}^{M}((1-\om)u_m, [\tilde{g}_{kh}]_m)_{\Om}= J_1 + J_2+J_3,
\end{aligned}
\end{equation}
where again in the sum with jumps we included the last term by setting $\tilde{g}_{kh,M+1} = 0$ and defining
consequently $[\tilde{g}_{kh}]_M = -\tilde{g}_{kh,M}$.
For $J_1$, adding and subtracting $(R_h(1-\om)u,\pa_t\tilde{g}_{kh})_{\IOm}$, we obtain
$$
J_1=-\sum_{m=1}^{M}(R_h(1-\om)u,\pa_t\tilde{g}_{kh})_{\ImOm}+\sum_{m=1}^{M}((I-R_h)(1-\om)u,\pa_t\tilde{g}_{kh})_{\ImOm}=J_{11}+J_{12}.
$$
Using that $\si^{-\frac{N}{2}}\le Cd^{-\frac{N}{2}}$ on $\Om\backslash B_d$ and $(1-\om)\le 1$, we obtain
\[
\begin{aligned}
J_{11}&=\sum_{m=1}^{M}(\na((1-\om)u),\na\Delta_h^{-1}\pa_t\tilde{g}_{kh})_{\ImOm}\\
&=\sum_{m=1}^{M}(\si^{-\frac{N}{2}}\na((1-\om)u),\si^{\frac{N}{2}}\na\Delta_h^{-1}\pa_t\tilde{g}_{kh})_{\ImOm}\\
&\le Cd^{-\frac{N}{2}}\sum_{m=1}^{M}\|\na((1-\om)u)\|_{L^\infty(I_m; L^2(\Om))}\|\si^{\frac{N}{2}}\na\Delta_h^{-1}\pa_t\tilde{g}_{kh}\|_{L^1(I_m; L^2(\Om))}\\
&\le Cd^{-\frac{N}{2}}\left(d^{-1}\|u\|_{L^\infty(I; L^2(\Om))}+\|\na u\|_{L^\infty(I; L^2(\Om))}\right)\sum_{m=1}^{M}\|\si^{\frac{N}{2}}\na\Delta_h^{-1}\pa_t\tilde{g}_{kh}\|_{L^1(I_m; L^2(\Om))}.
\end{aligned}
\]
Applying the discrete maximal parabolic regularity result from Lemma~\ref{lemma: fully discrete_maximal_parabolic} with respect to $\norm{\sigma^{\frac{N}{2}}\na\Delta_h^{-1}(\cdot)}_{L^2(\Omega)}$ we get
\begin{equation}\label{eq:J11_local}
\begin{aligned}
J_{11}&\le C\lk \lh^{\frac{N-1}{N}}d^{-\frac{N}{2}}\left(d^{-1}\|u\|_{L^\infty(I; L^2(\Om))}+\|\na u\|_{L^\infty(I; L^2(\Om))}\right)\|\si^{\frac{N}{2}}\na \Delta_h^{-1}P_hD\tilde{\delta}\|_{L^2(\Om)}\|\tilde{\theta}\|_{L^1(I_M)}\\
&\le C\lk\lh^{\frac{N-1}{N}+\half} d^{-\frac{N}{2}}\left(d^{-1}\|u\|_{L^\infty(I; L^2(\Om))}+\|\na u\|_{L^\infty(I; L^2(\Om))}\right),
\end{aligned}
\end{equation}
where in the last step we used Lemma~\ref{lemma: Delta_h-1 for Di delta}, Lemma~\ref{lemma:sigma_delta} and the fact that $\|\tilde{\theta}\|_{L^1(I_M)} \le C$.

The estimate for $J_{12}$ is slightly more involved since $R_h$ is a global operator. Put $\psi=(1-\om)u$, then pointwise in time we obtain
$$
((I-R_h)\psi,\pa_t\tilde{g}_{kh})_{\Om}=((I-R_h)\psi,\pa_t\tilde{g}_{kh})_{B_{d/2}}+((I-R_h)\psi,\pa_t\tilde{g}_{kh})_{\Om\backslash B_{d/2}}=I_1+I_2.
$$
Using local pointwise error estimates~\cite{AHSchatz_LBWahlbin_1977a}, the fact that $\psi$ is  supported on $\Om\backslash B_{d}$,  and the standard error estimate for $R_h$ we have
$$
\begin{aligned}
I_1&\le \|(I-R_h)\psi\|_{L^\infty{(B_{d/2})}}\|\pa_t\tilde{g}_{kh}\|_{L^1{(B_{d/2})}}\\
&\le C\left(\lh\|\psi\|_{L^\infty{(B_d)}}+d^{-\frac{N}{2}}\|(I-R_h)\psi\|_{L^2{(\Om)}}\right)\|\pa_t\tilde{g}_{kh}\|_{L^1{(\Om)}}\\
&\le Chd^{-\frac{N}{2}}\|\na\psi\|_{L^2{(\Om)}}\|\pa_t\tilde{g}_{kh}\|_{L^1{(\Om)}}\le Chd^{-\frac{N}{2}}\left(d^{-1}\|u\|_{L^2(\Om)}+\|\na u\|_{L^2(\Om)}\right)\|\pa_t\tilde{g}_{kh}\|_{L^1{(\Om)}}.
\end{aligned}
$$
Using that $\si\geq Cd$ on $\Om\backslash B_{d/2}$ we have for $I_2$:
$$
\begin{aligned}
I_2&=(\si^{-\frac{N}{2}}(I-R_h)\psi,\si^{\frac{N}{2}}\pa_t\tilde{g}_{kh})_{\Om\backslash B_{d/2}}\le Cd^{-\frac{N}{2}}\|(I-R_h)\psi\|_{L^2{(\Om)}}\|\si^{\frac{N}{2}}\pa_t\tilde{g}_{kh}\|_{L^2{(\Om)}}\\
&\le Chd^{-\frac{N}{2}}\|\na\psi\|_{L^2{(\Om)}}\|\si^{\frac{N}{2}}\pa_t\tilde{g}_{kh}\|_{L^2{(\Om)}}\le Chd^{-\frac{N}{2}}\left(d^{-1}\|u\|_{L^2(\Om)}+\|\na u\|_{L^2(\Om)}\right)\|\si^{\frac{N}{2}}\pa_t\tilde{g}_{kh}\|_{L^2{(\Om)}}.
\end{aligned}
$$
Combining estimates for $I_1$ and $I_2$ and using discrete maximal parabolic regularity from Lemma~\ref{lemma: fully discrete_maximal_parabolic} with respect to the $L^1(\Omega)$ norm and $\ltwonorm{\si^{\frac{N}{2}}(\cdot)}$, we obtain
\begin{equation}\label{eq:J12_local}
\begin{aligned}
J_{12}&\le C hd^{-\frac{N}{2}}\left(d^{-1}\|u\|_{L^\infty(I; L^2(\Om))}+\|\na u\|_{L^\infty(I; L^2(\Om))}\right) \sum_{m=1}^{M}\left(\|\pa_t\tilde{g}_{kh}\|_{L^1{(I_m\times\Om)}}+ \|\si^{\frac{N}{2}}\pa_t\tilde{g}_{kh}\|_{L^1(I_m; L^2{(\Om)})}\right)\\
&\le C\lk hd^{-\frac{N}{2}}\left(d^{-1}\|u\|_{L^\infty(I; L^2(\Om))}+\|\na u\|_{L^\infty(I; L^2(\Om))}\right) \|\tilde{\theta}\|_{L^1(I_M)}\times\\
&\left(\|P_h D\tilde{\delta}\|_{L^1{(\Om)}}+ \lh \|\si^{\frac{N}{2}}P_h D\tilde{\delta}\|_{L^2{(\Om)}}\right)\le C \lk \lh d^{-\frac{N}{2}}\left(d^{-1}\|u\|_{L^\infty(I; L^2(\Om))}+\|\na u\|_{L^\infty(I; L^2(\Om))}\right),
\end{aligned}
\end{equation}
where in the last step we used Lemma~\ref{lemma:sigma_delta}.
Thus, combining estimates for $J_{11}$ and $J_{12}$ we obtain
$$
J_1\le C\lk\lh^{\frac{N-1}{N}+\half} d^{-\frac{N}{2}}\left(d^{-1}\|u\|_{L^\infty(I; L^2(\Om))}+\|\na u\|_{L^\infty(I; L^2(\Om))}\right) .
$$
To estimate $J_2$, we use the H\"{o}lder inequality, Lemma \ref{lemma: fully discrete_maximal_parabolic}, and Lemma \ref{lemma: Delta_h-1 for Di delta}, to obtain
\begin{equation}\label{eq:J2_local}
\begin{aligned}
J_2&=(\si^{-\frac{N}{2}}\na((1-\om)u), \si^{\frac{N}{2}} \na \tilde{g}_{kh})_{\IOm}\\
&\le Cd^{-\frac{N}{2}}\|\na((1-\om)u)\|_{L^\infty(I; L^2(\Om))}\|\si^{\frac{N}{2}} \na \tilde{g}_{kh}\|_{L^1(I; L^2(\Om))}\\
&\le C\lk \lh^{\frac{N-1}{N}}d^{-\frac{N}{2}}\left(d^{-1}\|u\|_{L^\infty(I; L^2(\Om))}+\|\na u\|_{L^\infty(I; L^2(\Om))}\right)\|\si^{\frac{N}{2}}\na \Delta_h^{-1}P_h D\tilde{\delta}\|_{L^2(\Om)}\|\tilde{\theta}\|_{L^1(I_M)}\\
&\le C\lk\lh^{\frac{N-1}{N}+\half} d^{-\frac{N}{2}}\left(d^{-1}\|u\|_{L^\infty(I; L^2(\Om))}+\|\na u\|_{L^\infty(I; L^2(\Om))}\right).
\end{aligned}
\end{equation}
Similarly to $J_1$, to estimate $J_3$,
we, add and subtract $(R_h(1-\om)u,[\tilde{g}_{kh}]_m)_{\Om}$, to obtain
$$
J_3=-\sum_{m=1}^{M}(R_h((1-\om)u_m),[\tilde{g}_{kh}]_m)_{\Om}+\sum_{m=1}^{M}((I-R_h)((1-\om)u_m),[\tilde{g}_{kh}]_m)_{\Om}=J_{31}+J_{32}.
$$
Similarly to $J_{11}$, using that $\si^{-\frac{N}{2}}\le Cd^{-\frac{N}{2}}$ on $\Om\backslash B_d$ and $(1-\om)\le 1$, we obtain
\begin{equation}\label{eq:J31_local}
\begin{aligned}
J_{31}&=\sum_{m=1}^{M}(\na((1-\om)u),\na\Delta_h^{-1}[\tilde{g}_{kh}]_m)_{\Om}\\
&=\sum_{m=1}^{M}(\si^{-\frac{N}{2}}\na((1-\om)u),\si^{\frac{N}{2}}\na\Delta_h^{-1}[\tilde{g}_{kh}]_m)_{\Om}\\
&\le Cd^{-\frac{N}{2}}\sum_{m=1}^{M}\|\na((1-\om)u)\|_{L^\infty(I_m; L^2(\Om))}\|\si^{\frac{N}{2}}\na\Delta_h^{-1}[\tilde{g}_{kh}]_m\|_{L^2(\Om)}\\
&\le Cd^{-\frac{N}{2}}\left(d^{-1}\|u\|_{L^\infty(I; L^2(\Om))}+\|\na u\|_{L^\infty(I; L^2(\Om))}\right)\sum_{m=1}^{M}\|\si^{\frac{N}{2}}\na\Delta_h^{-1}[\tilde{g}_{kh}]_m\|_{L^2(\Om)}\\
&\le C\lk \lh^{\frac{N-1}{N}} d^{-\frac{N}{2}}\left(d^{-1}\|u\|_{L^\infty(I; L^2(\Om))}+\|\na u\|_{L^\infty(I; L^2(\Om))}\right)\|\si^{\frac{N}{2}}\na \Delta_h^{-1}P_h D\tilde{\delta}\|_{L^2(\Om)}\|\tilde{\theta}\|_{L^1(I_M)}\\
&\le C\lk\lh^{\frac{N-1}{N}+\half} d^{-\frac{N}{2}}\left(d^{-1}\|u\|_{L^\infty(I; L^2(\Om))}+\|\na u\|_{L^\infty(I; L^2(\Om))}\right).
\end{aligned}
\end{equation}
Similarly to $J_{12}$ we also obtain
$$
J_{32}\le C\lk \lh d^{-\frac{N}{2}}\left(d^{-1}\|u\|_{L^\infty(I; L^2(\Om))}+\|\na u\|_{L^\infty(I; L^2(\Om))}\right).
$$
Combining the estimates for $J_1$, $J_2$, and $J_3$, and taking supremum over all partial derivatives, we conclude that
\begin{multline*}
|\na u_{kh}(\tilde{t},x_0)|\le C\ell_k \ell_h\Bigl( d^{-1}\norm{u}_{L^\infty(I\times B_{2d})}+\norm{\na u}_{L^\infty(I\times B_{2d})}\\
+ d^{-\frac{N}{2}}\left(d^{-1}\|u\|_{L^\infty(I; L^2(\Om))}+\|\na u\|_{L^\infty(I; L^2(\Om))}\right)\Bigr).
\end{multline*}
Using that the dG($q$)cG($r$) method is invariant on $\Xkh$, by replacing $u$ and $u_{kh}$ with $u-\chi$ and $u_{kh}-\chi$ for any $\chi\in \Xkh$, we obtain Theorem~\ref{thm:local_best_approx}.

\section*{Acknowledgments}
The authors would like to thank Dominik Meidner for the careful reading of the manuscript and providing valuable suggestions that help to improve the presentation of the paper.


\bibliography{lit_W1_inf}
\bibliographystyle{siam}

\end{document}